\documentclass[a4paper]{amsart}

\usepackage[english]{babel}

\usepackage{amssymb}
\usepackage{mathrsfs}
\usepackage{hyperref}
\usepackage{xcolor}
\usepackage{eucal}

\usepackage{amssymb}
\usepackage{amsmath}
\usepackage{amsthm}
\usepackage{imakeidx}
\usepackage{breqn}

\usepackage{diagbox}

\usepackage{verbatim}

\usepackage{fancyhdr}

\usepackage{tikz-cd}

\usepackage[utf8]{inputenc}

\usepackage{graphicx}

\usepackage{xcolor}

\usepackage[style=alphabetic, url=false, doi=false, backend=bibtex]{biblatex}

\theoremstyle{definition}

\newtheorem{mydef}{Definition}[section]

\newtheorem*{myass}{Assumption}
\newtheorem{myque}[mydef]{Question}
\newtheorem*{myack}{Acknowledgments}

\theoremstyle{remark}

\newtheorem{mybem}[mydef]{Remark}
\newtheorem{myex}[mydef]{Example}

\theoremstyle{plain}

\newtheorem{mysen}[mydef]{Theorem}
\newtheorem{mylem}[mydef]{Lemma}
\newtheorem{mypro}[mydef]{Proposition}
\newtheorem{myfact}[mydef]{Fact}
\newtheorem{myclaim}{Claim}
\newtheorem*{mysenx}{Theorem}

\newtheorem*{mysclaim}{Subclaim}

\numberwithin{mydef}{section}

\DeclareMathOperator{\cof}{cof}
\DeclareMathOperator{\dom}{dom}

\DeclareMathOperator{\im}{im}

\DeclareMathOperator{\cf}{cf}

\DeclareMathOperator{\Coll}{Coll}

\DeclareMathOperator{\stem}{stem}
\DeclareMathOperator{\osucc}{osucc}

\newcommand{\dC}{\mathbb{C}}

\newcommand{\dG}{\mathbb{G}}

\newcommand{\dI}{\mathbb{I}}

\newcommand{\dL}{\mathbb{L}}

\newcommand{\dP}{\mathbb{P}}
\newcommand{\dQ}{\mathbb{Q}}

\newcommand{\uhr}{\upharpoonright}

\newcommand{\ZFC}{\mathsf{ZFC}}
\newcommand{\GCH}{\mathsf{GCH}}
\newcommand{\PCF}{\mathsf{PCF}}

\newcommand{\LIP}{\mathsf{LIP}}

\title[Failure of Approachability at $\aleph_{\omega+1}$ for any Cofinality]{Failure of Approachability at the Successor of the first Singular for any Cofinality} 

\author{Hannes Jakob and Maxwell Levine} 

\subjclass[2020]{Primary: 03E05, Secondary: 03E04, 03E35, 03E55} 

\date{\today}

\begin{document}
	
	
	\keywords{} 
	
	
	\begin{abstract}
		We solve two long-standing open problems regarding the combinatorics of $\aleph_{\omega+1}$. We answer a question of Shelah by showing that it is consistent for any $n\geq 1$ that $\mathsf{GCH}$ holds and there is a stationary set of points of cofinality $\aleph_n$ which is not in the approachability ideal. As a corollary, we obtain a model where the notions of goodness and approachability are distinct for stationarily many points of cofinality $\aleph_1$, answering an open question of Cummings, Foreman, and Magidor.
	\end{abstract}
	
	\maketitle
	
	The concept of independence is fundamental to set-theoretic research. It refers to the phenomenon that many statements are neither proven nor refuted by the axioms of $\ZFC$. One of the first independence proofs was carried out by Paul Cohen, who showed using his method of forcing that the Continuum Hypothesis -- the statement that $2^{\aleph_0}=\aleph_1$ -- is independent. Once it became apparent that independence was an essential consideration in the study of cardinal exponentiation, an important distinction arose between regular and singular cardinals. A cardinal $\kappa$ is regular if it cannot be written as $\bigcup_{i<\tau}\lambda_i$, where $\tau$ and all $\lambda_i$ are less than $\kappa$; otherwise it is singular. For the continuum function on the regular cardinals, William Easton was able to show that it is constrained only by the properties that it is increasing and that $\cf(2^{\kappa})>\kappa$.
	
	However, it turns out that for the singular cardinals, a surprising number of significant statements are determined by $\ZFC$: Saharon Shelah introduced $\PCF$ theory precisely to study the behavior of non-regular cardinals like $\aleph_{\omega}$. As the first singular cardinal, it is a natural object to study. Using his $\PCF$ framework, he was able to show the remarkable result that if $\aleph_{\omega}$ is a strong limit, $2^{\aleph_{\omega}}<\aleph_{\omega_4}$. This is one example showing that the behavior of the continuum function on the singular cardinals is subject to additional non-trivial constraints. Another example of the rigidity of the continuum function for singular cardinals is a result by Jack Silver, who showed that, if $\gamma$ is a singular cardinal of uncountable cofinality and $2^{\alpha}=\alpha^+$ for stationarily many $\alpha<\gamma$, then $2^{\gamma}=\gamma^+$ as well (see \cite[Theorem 8.12]{JechSetTheory}).
	
	Shelah's proof of the bound on $2^{\aleph_{\omega}}$ relied on two canonical invariants which he isolated: The set of \emph{good points} -- those points which behave in an orderly manner with respect to a cofinal sequence through $\prod_n\aleph_n$ -- and the set of \emph{approachable points} -- those points which are limits of particularly strong elementary substructures of the universe.
	
	Shelah showed that any ordinal of cofinality $>\!\aleph_3$ is good (which is deeply connected to his bound on $2^{\aleph_\omega}$) and that any approachable point must also be a good point. However, it was not known if the latter implication can be reversed or if there can be non-approachable points of cofinality $>\!\aleph_3$. It was even speculated by Cummings, Foreman and Magidor (see \cite[Page 2]{CumForeMagCanonicalOne}) that it might be the case that any good point must also be approachable. From that fact it would of course follow immediately that there are no non-approachable points of cofinality $>\!\aleph_3$.
	
	These questions are challenging because we lack a diversity of techniques for obtaining non-approachable points. Shelah showed that there can be a stationary set of non-approachable points with cofinality $\aleph_1$ by starting with a supercompact cardinal $\kappa$, finding a regular $\gamma<\kappa$ such that there are stationarily many non-approachable points of cofinality $\gamma^{+\omega+1}$ in $\kappa^{+\omega+1}$, and then collapsing $\gamma^{+\omega}$ to $\aleph_0$ (thus turning $\gamma^{+\omega+1}$ into $\aleph_1$) and $\kappa$ to $\aleph_2$. In the resulting model, $\kappa^{+\omega+1}$ becomes $\aleph_{\omega+1}$ and a straightforward argument shows that there are still stationarily many $\alpha\in\kappa^{+\omega+1}$ (which becomes $\aleph_{\omega+1}$ in the forcing extension) of cofinality $\gamma^{+\omega+1}$ (which becomes $\aleph_1$) that are not approachable. However, these techniques are isolated to $\aleph_1$ and the non-approachable points obtained in this way are also not good.
	
	In this paper, we aim to remedy the situation. We will provide a forcing poset which turns $\aleph_{\omega+1}$ into a non-approachable ordinal while (1) also making it a good point and (2) collapsing it to have size $\aleph_n$ for any desired $n$. This forcing will be an iteration of a variant of Namba forcing together with the Levy-collapse and have two ostensibly conflicting properties: A weak form of the approximation property without adding any new functions from $\aleph_0$ into any $\aleph_n$. It is precisely this tension which allows us to show that the forcing does not make $\aleph_{\omega+1}$ approachable.
	
	Moreover, our variant of Namba forcing will be a Prikry-type poset and thus iterable using a technique of Magidor. When starting from sufficiently large cardinals, we obtain, for any $n\geq 1$, a model where $\GCH$ holds and there are stationarily many points in $\aleph_{\omega+1}\cap\cof(\aleph_n)$ which are not approachable. As a corollary (but also a consequence of our proof) we obtain that there can consistently be stationarily many points (even of cofinality $\aleph_1$) which are good but not approachable:
	
	\begin{mysenx}
		Assume $\GCH$ holds and $(\kappa_k)_{k\in\omega}$ is an increasing sequence of supercompact cardinals. Let $n\in\omega$, $n\geq 1$. There is a forcing extension in which $\GCH$ holds, $\kappa_0=\aleph_{n+1}$, $(\sup_k\kappa_k)^+=\aleph_{\omega+1}$ and there are stationarily many $\gamma\in\aleph_{\omega+1}\cap\cof(\aleph_n)$ which are good, but not approachable.
	\end{mysenx}
	
	We note that, since $\square_{\aleph_{\omega}}$ necessarily fails in our model, at least some large cardinals are necessary (see \cite{SargsyanSquareFailure}).
	
	The paper is organized as follows. In Section 1, we give a more in-depth introduction into the combinatorics at $\aleph_{\omega+1}$ and our results. In Section 2, we provide proofs of a few well-known lemmas about iterations of Prikry-type forcings. In Section 3, we show the consistency of the existence of a certain sequence of ideals that will be necessary for the definition and properties of our main forcing poset, which will occur in Section 4. In Section 5, we prove the main result. In the last section we modify the proof of our main theorem to obtain the existence of a model where all scales are good but the approachability property fails.
	
	We assume that the reader is familiar with the basics of cardinal arithmetic, forcing and large cardinals. Good introductory material can be found in the textbooks of Jech (see \cite{JechSetTheory}) and Kunen (see \cite{KunenSetTheory}). A more in-depth introduction into the combinatorics at singular cardinals can be found in the chapters of Abraham-Magidor (see \cite{AbrahamHandbook}) or Eisworth (see \cite{EisworthHandbook}) in the Handbook of Set Theory. In the proof of the main theorem, we will lift a ground-model embedding. Results related to this technique can be found in Cummings' chapter in the Handbook of Set Theory (see \cite[Section 9]{CummingsHandbook}).
	
	\begin{myack}
		The first author would like to thank Chris Lambie-Hanson for many illuminating conversations surrounding this problem during a recent research visit.
	\end{myack}
	
	\section{Combinatorics at $\aleph_{\omega+1}$}
	
	In this section, we will explain the notions of goodness and approachability and state standard facts that we will use freely throughout the paper. For simplicity, we will focus on the special case where our singular cardinal has countable cofinality, but the definitions and results can be generalized to singulars of any cofinality.
	
	\begin{mydef}
		Let $\delta$ be a singular cardinal with countable cofinality. Let $(\delta_n)_{n\in\omega}$ be an increasing sequence of regular cardinals converging to $\delta$. A \emph{$(\delta^+,(\delta_n)_{n\in\omega})$-scale} is a sequence $(f_{\alpha})_{\alpha<\delta^+}$ such that:
		\begin{enumerate}
			\item For all $\alpha<\delta^+$, $f_{\alpha}\in\prod_{n\in\omega}\delta_n$.
			\item For all $\alpha<\beta<\delta^+$, $f_{\alpha}<^*f_{\beta}$, i.e. there is $k\in\omega$ such that $f_{\alpha}(n)<f_{\beta}(n)$ for all $n\geq k$.
			\item For all $g\in\prod_{n\in\omega}\delta_n$ there is $\alpha<\delta^+$ such that $g<^*f_{\alpha}$.
		\end{enumerate}
		
		Given a $(\delta^+,(\delta_n)_{n\in\omega})$-scale $(f_{\alpha})_{\alpha<\delta^+}$, we say that an ordinal $\gamma<\delta^+$ is \emph{good for $(f_{\alpha})_{\alpha<\delta^+}$} if there exists an unbounded $A\subseteq\gamma$ and $k\in\omega$ such that for all $n\geq k$, the sequence $(f_{\alpha}(n))_{\alpha\in A}$ is strictly increasing.
	\end{mydef}
	
	Shelah showed that whenever $\delta$ is singular with countable cofinality, there is an increasing sequence $(\delta_n)_{n\in\omega}$ of regular cardinals converging to $\delta$ such that there exists a $(\delta^+,(\delta_n)_{n\in\omega})$-scale (see \cite[Theorem 3.53]{EisworthHandbook}). Moreover, if $\delta=\aleph_{\omega}$, there is a maximal choice of $(\delta_n)_{n\in\omega}$ modulo finite sets (this is folklore). Additionally, whenever $(f_{\alpha})_{\alpha<\delta^+}$ and $(f_{\alpha}')_{\alpha<\delta^+}$ are both $(\delta^+,(\delta_n)_{n\in\omega})$-scales, there is a club $C\subseteq\delta^+$ such that any $\gamma\in C$ is good for $(f_{\alpha})_{\alpha<\delta^+}$ if and only if it is good for $(f_{\alpha}')_{\alpha<\delta^+}$, since there is a club of points $\gamma\in\delta^+$ such that $\{f_{\alpha}\;|\;\alpha<\gamma\}$ is unbounded in $\{f_{\alpha}'\;|\;\alpha<\gamma\}$ and vice versa. Ergo it makes sense to speak of ``good points'' without referring to the exact scale, knowing that this is well-defined modulo club sets.
	
	The following characterization of goodness is often easier to verify (see \cite[Theorem 3.50]{EisworthHandbook}):
	
	\begin{myfact}
		Let $(f_{\alpha})_{\alpha<\delta^+}$ be a $(\delta^+,(\delta_n)_{n\in\omega})$-scale and $\gamma<\delta^+$. Then $\gamma$ is good for $(f_{\alpha})_{\alpha<\delta^+}$ if and only if there exists $h\in\prod_{n\in\omega}\delta_n$ such that:
		\begin{enumerate}
			\item For almost all $n\in\omega$, $\cf(h(n))=\cf(\gamma)$.
			\item For all $\alpha<\gamma$, $f_{\alpha}<^*h$.
			\item For all $g\in\prod_{n\in\omega}\delta_n$ with $g<^*h$ there is some $\alpha<\gamma$ such that $g<^*f_{\alpha}$.
		\end{enumerate}
	\end{myfact}
	
	Points (2) and (3) state that $h$ is an \emph{exact upper bound of $(f_{\alpha})_{\alpha<\gamma}$}.
	
	The approachability ideal was introduced by Shelah (see \cite{ShelahApproachability}) in order to obtain indestructibility results for stationary sets under sufficiently closed forcing notions.
	
	\begin{mydef}
		Let $\lambda$ be a cardinal. A set $S$ is in the \emph{approachability ideal} $I[\lambda^+]$ if there is a club $C\subseteq\lambda^+$ and a sequence $(a_{\alpha})_{\alpha<\lambda^+}$ of elements of $[\lambda^+]^{<\lambda}$ such that whenever $\gamma\in S\cap C$, there is an unbounded set $A\subseteq\gamma$ with ordertype $\cf(\gamma)$ such that $A\cap\beta\in\{a_{\alpha}\;|\;\alpha<\gamma\}$ for all $\beta<\gamma$.
	\end{mydef}
	
	If $\lambda^{<\lambda}\leq\lambda^+$ (so e.g. if $2^{\lambda}=\lambda^+$) there is a single sequence $(a_{\alpha})_{\alpha<\lambda^+}$ enumerating all of $[\lambda^+]^{<\lambda}$. In this case, the approachability ideal is generated by a single set modulo the club filter. So again it makes sense to speak of ``approachable points'' without referring to a specific sequence $(a_{\alpha})_{\alpha<\lambda^+}$, knowing that this is well-defined modulo club sets.
	
	The approachability ideal behaves quite differently at successors of regulars when compared to successors of singulars: For regular cardinals $\lambda$, $\lambda$ is often referred to as the \emph{critical cofinality}, since the set $\lambda^+\cap\cof(<\!\lambda)$ is always in $I[\lambda^+]$. Due to this, the approachability ideal is completely determined by the membership of subsets of $\lambda^+\cap\cof(\lambda)$. For singular cardinals $\lambda$ however, there are many more possibilities: $\ZFC$ only implies that the set $\lambda^+\cap\cof(\leq\!\cf(\lambda))$ is in $I[\lambda^+]$ and since there are many cardinals between $\cf(\lambda)$ and $\lambda$ there is a great potentiality regarding the behavior of $I[\lambda^+]$.
	
	The notions of goodness and approachability share a deep connection: Shelah showed that any approachable point is also good. Additionally, he was the first to show that consistently there can be stationarily many non-approachable points in $\aleph_{\omega+1}$ of cofinality $\aleph_1$ (see \cite{Shelah88}, for a proof see \cite{ShelahSuccSingCard}). In the first paper, he raised the following question:
	
	\begin{myque}
		Is $\GCH+\{\delta<\aleph_{\omega+1}\;|\;\cf(\delta)>\aleph_1\}\notin I[\aleph_{\omega+1}]$ consistent with $\ZFC$?
	\end{myque}
	
	A related question appears in a survey by Foreman (see \cite{ForemanSurvey}):
	
	\begin{myque}
		Can there be a stationary set of good points which are not approachable?
	\end{myque}
	
	These questions have the following connection: It can be shown that any point of cofinality $>\!2^{\aleph_0}$ is good (see \cite[Corollary 4.61]{EisworthHandbook}). Additionally, Shelah showed that, regardless of any cardinal arithmetic, any point of cofinality $>\!\aleph_3$ is good as well (see \cite[Section 2.1]{AbrahamHandbook}). So answering the question of Shelah would furnish us with a model where there are necessarily stationarily many points of cofinality say $\aleph_n$ which are good, but not approachable. One could then even use the simple Levy collapse of $\aleph_{n-1}$ to $\aleph_0$ in order to obtain a model where there are stationarily many non-approachable points of cofinality $\aleph_1$ which are good.
	
	A very powerful characterization of approachability, specifically at successors of singular cardinals, can be obtained using normal, subadditive colorings. For ease of notation, we will regard $[\kappa]^2$ as the set of all pairs $(\alpha,\beta)$ with $\alpha<\beta<\kappa$.
	
	\begin{mydef}
		Let $\delta$ be a limit cardinal with countable cofinality and let $(\delta_n)_{n\in\omega}$ be an ascending sequence of regular cardinals converging to $\delta$. Let $d\colon[\delta^+]^2\to\omega$ be a coloring. Then we say that:
		\begin{enumerate}
			\item $d$ is \emph{normal} if for any $n\in\omega$,
			$$\sup_{\alpha\in\delta^+}|\{\beta<\alpha\;|\;d(\beta,\alpha)\leq n\}|<\delta$$
			\item $d$ is \emph{$(\delta_n)_{n\in\omega}$-normal} if for any $n\in\omega$ and $\alpha\in\delta^+$,
			$$|\{\beta<\alpha\;|\;d(\beta,\alpha)\leq n\}|\leq\delta_n$$
			\item $d$ is \emph{subadditive} if for any $\alpha<\beta<\gamma<\delta^+$,
			$$d(\alpha,\gamma)\leq\max\{d(\alpha,\beta),d(\beta,\gamma)\}$$
		\end{enumerate}
	\end{mydef}
	
	The notion of $(\delta_n)_{n\in\omega}$-normality is not really necessary for the arguments but it does streamline the notation, so we have chosen to define it. It can be shown easily that whenever $\delta$ is singular of countable cofinality and $(\delta_n)_{n\in\omega}$ is a sequence of regular cardinals converging to $\delta$, there exists a $(\delta_n)_{n\in\omega}$-normal subadditive coloring on $\delta^+$.
	
	\begin{mydef}
		Let $\delta$ be a limit cardinal with countable cofinality and $d\colon[\delta^+]^2\to\omega$ a normal subadditive coloring. Then an ordinal $\alpha\leq\delta^+$ with $\cf(\alpha)>\omega$ is \emph{$d$-approachable} if there is an unbounded $A\subseteq\alpha$ and $n\in\omega$ such that $d(\beta_0,\beta_1)\leq n$ for any $\beta_0<\beta_1$, both in $A$. We let $S(d)$ consist of all those $\alpha<\delta^+$ which are $d$-approachable.
	\end{mydef}
	
	Whenever $d\colon[\delta^+]^2\to\omega$ is a normal subadditive coloring, $\delta^+$ itself can never be $d$-approachable. However, we have chosen to include it in the definition since it might become $d$-approachable in forcing extensions where it is no longer a cardinal (so we can say ``$\delta^+$ does not become $d$-approachable'').
	
	Interestingly, there is the following straightforward link between scales and colorings:
	
	\begin{myex}
		Let $\vec{f}=(f_{\alpha})_{\alpha<\delta^+}$ be a $(\delta^+,(\delta_n)_{n\in\omega})$-scale. The \emph{scale coloring} $d_{\vec{f}}$ is defined by
		$$d_{\vec{f}}(\alpha,\beta):=\min\{n\in\omega\;|\;\forall k\geq n(f_{\alpha}(k)<f_{\beta}(k))\}$$
		this coloring is subadditive (if $f_{\alpha}(k)<f_{\gamma}(k)$, it follows that either $f_{\alpha}(k)<f_{\beta}(k)$ or $f_{\beta}(k)<f_{\gamma}(k)$), but not normal.
		
		Clearly, a point $\gamma$ is good if for $(f_{\alpha})_{\alpha<\delta^+}$ if and only if it is $d_{\vec{f}}$-approachable.
	\end{myex}
	
	This highlights an interesting problem when trying to obtain non-approachable good points: In order to make a point $\gamma$ non-$d$-approachable, we have to make sure that there is no unbounded subset of $\gamma$ on which $d$ is bounded. However, in order to make $\gamma$ good, there needs to be an unbounded subset of $\gamma$ on which $d_{\vec{f}}$ is bounded. Therefore we have to obtain preservation theorems which incorporate the normality of the coloring.
	
	We will also need the following fact about $d$-approachability (see \cite[Remark 28]{ShelahSuccSingCard}):
	
	\begin{myfact}\label{ApproachRefinement}
		Let $\delta$ be a limit cardinal with countable cofinality and let $d\colon[\delta^+]^2\to\omega$ be a subadditive coloring. Let $\alpha<\delta^+$, $\cf(\alpha)>\omega$. Then $\alpha$ is $d$-approachable if and only if whenever $B\subseteq\alpha$ is unbounded, there is $B'\subseteq B$ unbounded and $n\in\omega$ such that $d(\beta_0,\beta_1)\leq n$ for any $\beta_0<\beta_1$, both in $B'$.
	\end{myfact}
	
	\begin{proof}
		Let $A\subseteq\alpha$ be unbounded such that $d\uhr[A]^2$ is bounded by some $n$. By refining $A$ and $B$ if necessary, we can assume that $A=\{\alpha_i\;|\;i<\cf(\alpha)\}$, $B=\{\beta_i\;|\;i<\cf(\alpha)\}$ and for any $i$, $\alpha_i<\beta_i<\alpha_{i+1}$. For $i<\cf(\alpha)$, let
		$$n_i:=\max\{d(\alpha_i,\beta_i),d(\beta_i,\alpha_{i+1}),n\}$$
		Then there is an unbounded $X\subseteq\cf(\alpha)$ such that $n_i=n^*$ for all $i\in X$. In particular, $\{\beta_i\;|\;i\in X\}$ is unbounded in $\alpha$. Moreover, for any $i<j$, both in $X$, we have
		\begin{align*}
			d(\beta_i,\beta_j) & \leq\max\{d(\beta_i,\alpha_{i+1}),d(\alpha_{i+1},\beta_j)\} \\
			&\leq\max\{d(\beta_i,\alpha_{i+1}),d(\alpha_{i+1},\alpha_j),d(\alpha_j,\beta_j)\} \\
			&\leq n^*
		\end{align*}
		so $d$ is bounded on $\{\beta_i\;|\;i\in X\}\subseteq B$.
	\end{proof}
	
	The following fact, which was first proven by Shelah (see \cite{Shelah88} and \cite[Corollary 3.35]{EisworthHandbook} for a more modern proof), connects normal subadditive colorings to the approachability ideal (recall that $S_{\omega}^{\delta^+}:=\delta^+\cap\cof(\omega)$):
	
	\begin{myfact}\label{ColoringApproach}
		Let $\delta$ be a singular strong limit cardinal with countable cofinality and let $d\colon[\delta^+]^2\to\omega$ be a normal subadditive coloring. Then $S(d)\cup S_{\omega}^{\delta^+}$ generates $I[\lambda]$ modulo nonstationary sets.
	\end{myfact}
	
	The upshot of the previous fact is that to show the existence of stationarily many non-approachable points it suffices to fix an arbitrary normal subadditive coloring $d$ and show that there exist stationarily many non-$d$-approachable points. Additionally, since the properties of being subadditive and normal are often preserved by forcing extensions, one can take such a coloring from the ground model.
	
	Our approach is therefore the following: We will introduce a forcing which collapses the current $\aleph_{\omega+1}$ to some arbitrary $\aleph_n$ while not making it $d$-approachable for any normal subadditive coloring $d$ from the ground model (but necessarily making it $e$-approachable for some non-normal colorings $e$). Then we iterate instances of this forcing with supercompact length and obtain the desired forcing extension.
	
	\section{Iterations of Prikry-Type Forcings}
	
	We will be using Namba-style forcings to obtain our consistency results. In many cases, these forcings derive their regularity properties from very abstract considerations (such as Shelah's \emph{$\dI$-condition}) and are typically only iterable with revised countable support. This is due to the fact that we need countable supports in order not to introduce too many ``tasks'' but also have to revise our supports in order to accomodate new countable sets which are not covered by any countable set from the ground model.
	
	In this work, we will use Namba-style forcings derived from particularly well-behaved ideals which function in a similar way to diagonal Prikry forcing (see \cite[Section 1.3]{GitikHandbook}) and derive their regularity properties from the Prikry property together with a sufficiently closed direct extension ordering. Due to this, we can iterate such forcings in a much simpler way. This material can be found in Gitik's chapter in the Handbook of Set Theory (see \cite{GitikHandbook}), but we have chosen to include the proofs as some of them are only sketched in the chapter.
	
	\begin{mydef}
		Let $(\dP,\leq)$ be a partial order and $\leq_0$ another partial order on $\dP$ such that $\leq$ refines $\leq_0$. We say that $(\dP,\leq,\leq_0)$ is a \emph{Prikry-type forcing} if for every $p\in\dP$ and every statement $\sigma$ in the forcing language there is $q\leq_0p$ which decides $\sigma$ (this is known as the \emph{Prikry property}).
	\end{mydef}
	
	It turns out that Prikry-type forcing is iterable. Since we need the iteration to be $\nu$-cc at Mahlo cardinals $\nu$, we have chosen to iterate with Easton support instead of full support (see \cite[Section 6.3]{GitikHandbook}). This involves a strange requirement where we allow only finitely many non-direct extensions in the support but arbitrarily many non-direct extensions outside of it. The reason for this will become apparent later in the proof of the Prikry property.
	
	\begin{mydef}
		Let $((\dP_{\alpha},\leq_{\alpha},\leq_{\alpha,0}),(\dot{\dQ}_{\alpha},\dot{\leq}_{\alpha},\dot{\leq}_{\alpha,0}))_{\alpha<\rho}$ be a sequence such that each $(\dP_{\alpha},\leq_{\alpha},\leq_{\alpha,0})$ is a poset and each $(\dot{\dQ}_{\alpha},\dot{\leq}_{\alpha},\dot{\leq}_{\alpha,0})$ is a $\dP_{\alpha}$-name for a Prikry-type poset. We define the statement ``$((\dP_{\alpha},\leq_{\alpha},\leq_{\alpha,0}),(\dot{\dQ}_{\alpha},\dot{\leq}_{\alpha},\dot{\leq}_{\alpha,0}))_{\alpha<\rho}$ is an Easton-support Magidor iteration of Prikry-type forcings of length $\rho$'' by induction on $\rho$.
		
		$((\dP_{\alpha},\leq_{\alpha},\leq_{\alpha,0}),(\dot{\dQ}_{\alpha},\dot{\leq}_{\alpha},\dot{\leq}_{\alpha,0}))_{\alpha<\rho}$ is an Easton-support Magidor iteration of Prikry-type forcings of length $\rho$ if $((\dP_{\alpha},\leq_{\alpha},\leq_{\alpha,0}),(\dot{\dQ}_{\alpha},\dot{\leq}_{\alpha},\dot{\leq}_{\alpha,0}))_{\alpha<\rho'}$ is an Easton-support Magidor iteration of Prikry-type forcings of length $\rho'$ for every $\rho'<\rho$ and moreover:
		\begin{enumerate}
			\item If $\rho=\rho'+1$, then $(\dP_{\rho},\leq_{\rho}):=(\dP_{\rho'},\leq_{\rho'})*(\dot{\dQ}_{\rho'},\dot{\leq}_{\rho'})$ and $(p',\dot{q}')\leq_{\rho,0}(p,\dot{q})$ if and only if $p'\leq_{\rho',0}p$ and $p'\Vdash\dot{q}'\dot{\leq}_{\rho',0}\dot{q}$.
			\item If $\rho$ is a limit, then $\dP_{\rho}$ consists of all functions $p$ on $\rho$ such that
			\begin{enumerate}
				\item For all $\alpha<\rho$, $p\uhr\alpha\in\dP_{\alpha}$,
				\item If $\rho$ is inaccessible and $|\dP_{\alpha}|<\rho$ for every $\alpha<\rho$, then there is some $\beta<\rho$ such that for all $\gamma\in(\beta,\rho)$, $p\uhr\gamma\Vdash p(\gamma)=1_{\dot{\dQ}_{\gamma}}$.
			\end{enumerate}
			and the following holds:
			\begin{enumerate}
				\item[(i)] $p'\leq_{\rho}p$ if and only if $p'\uhr\rho'\leq_{\rho'}p\uhr\rho'$ for every $\rho'<\rho$ and there exists a finite subset $b$ such that whenever $\rho'\notin b$ and $p\uhr\rho'\not\Vdash p(\rho')=1_{\dot{\dQ}_{\rho'}}$, then $p'\uhr\rho'\Vdash p'(\rho')\dot{\leq}_{\rho',0} p(\rho')$.
				\item[(ii)] $p'\leq_{\rho,0}p$ if and only if $p'\leq_{\rho}p$ and the set $b$ is empty.
			\end{enumerate} 
		\end{enumerate}
	\end{mydef}
	
	It turns out that this notion of an iteration of Prikry-type forcings preserves the Prikry property. We first prove the statement for the two-step iteration.
	
	\begin{mylem}\label{SuccStep}
		Assume $(\dP,\leq,\leq_0)$ is a Prikry-Type forcing and $(\dot{\dQ},\dot{\leq},\dot{\leq}_0)$ is a $\dP$-name for a Prikry-Type forcing. Then $\dP*\dot{\dQ}$ is a Prikry-type forcing.
	\end{mylem}
	
	\begin{proof}
		Let $\sigma$ be a statement in the forcing language of $\dP*\dot{\dQ}$ and let $(p,\dot{q})\in\dP*\dot{\dQ}$. Because $\dot{\dQ}$ is forced to be a Prikry-type forcing we can apply the maximum principle and find $\dot{q}'$ such that $(p,\dot{q}')\leq_0(p,\dot{q})$ and $p$ forces that $\dot{q}'$ decides $\sigma$. Now we can find $p'\leq_0 p$ which decides the statement ``$\dot{q}'\Vdash\sigma$''. Then either $p'\Vdash\dot{q}'\Vdash\sigma$ and so $(p',\dot{q}')\Vdash\sigma$ or $p'\Vdash\dot{q}'\not\Vdash\sigma$ in which case it follows that $p'\Vdash\dot{q}'\Vdash\neg\sigma$ since $\dot{q}'$ is forced to decide $\sigma$ and thus $(p',\dot{q}')\Vdash\neg\sigma$.
	\end{proof}
	
	We now show that the Prikry property is preserved by Easton-support Magidor iterations (see \cite[Lemma 2]{GitikHandbook} for the full support case):
	
	\begin{mylem}\label{PrikryIter}
		Let $\rho$ be an ordinal. Assume that $((\dP_{\alpha},\leq_{\alpha},\leq_{\alpha,0}),(\dot{\dQ}_{\alpha},\dot{\leq}_{\alpha},\dot{\leq}_{\alpha,0}))_{\alpha<\rho}$ is an Easton-support Magidor iteration of Prikry-type forcings of length $\rho$. Then $\dP_{\rho}$ has the Prikry property.
	\end{mylem}
	
	\begin{proof}
		We prove the statement by induction on $\rho$. In case $\rho=1$, it is easy and the successor step is Lemma \ref{SuccStep}. So assume that $\rho$ is a limit ordinal and the statement holds for all $\rho'<\rho$.
		
		For the first case, assume $\rho$ is not inaccessible or that there is $\alpha<\rho$ such that $|\dP_{\alpha}|\geq\rho$ (in other words, assume we are taking a full support limit at $\rho$). Let $p\in\dP_{\rho}$ and let $\sigma$ be a statement in the forcing language. Fix any unbounded and continuous function $f\colon\cf(\rho)\to\rho$ such that either $|\dP_{\alpha}|\geq\rho$ for every $\alpha\in\im(f)$ or no inaccessible cardinal is in the image of $f$ (again, in other words, assume we are taking full support limits at every point in the image of $f$). For any $\rho'<\rho''\leq\rho$, let $(\dot{\dP}_{\rho',\rho''},\dot{\leq}_{\rho',\rho''},\dot{\leq}_{(\rho',\rho''),0})$ be a $\dP_{\rho'}$-name for a forcing poset such that $\dP_{\rho''}\cong\dP_{\rho'}*\dot{\dP}_{\rho',\rho''}$. We regard $\dP_{\rho}$ as the full support limit of $(\dP_{f(\beta)},\dot{\dP}_{f(\beta),f(\beta+1)})_{\beta<\cf(\rho)}$. Moreover, for any $\beta<\cf(\rho)$, $(\dot{\dP}_{f(\beta),f(\beta+1)},\dot{\leq}_{f(\beta),f(\beta+1)},\dot{\leq}_{(f(\beta),f(\beta+1)),0})$ is forced to be forcing equivalent to an Easton-support Magidor iteration of Prikry-type forcings: Whenever $\nu\in(f(\beta),f(\beta+1))$ is inaccessible in the ground model and $|\dP_{\alpha}|<\nu$ for every $\alpha<\nu$, $\nu$ remains inaccessible in any extension by $\dP_{f(\beta)}$, since $|\dP_{f(\beta)}|<\nu$. Ergo, $(\dot{\dP}_{f(\beta),f(\beta+1)},\dot{\leq}_{f(\beta),f(\beta+1)},\dot{\leq}_{(f(\beta),f(\beta+1)),0})$ is forced to be dense with respect to both orderings in the Easton-support Magidor iteration of Prikry-type forcings with iterands ($(\dot{\dQ}_{\alpha})_{\alpha\in[f(\beta),f(\beta+1))}$ (for more details, see the discussion before \cite[Lemma 6.15]{GitikHandbook}). It follows from the inductive hypothesis that any $(\dot{\dP}_{f(\beta),f(\beta+1)},\dot{\leq}_{f(\beta),f(\beta+1)},\dot{\leq}_{(f(\beta),f(\beta+1)),0})$ is forced to have the Prikry property.
		
		So assume toward a contradiction that no $p'\leq_{\rho,0}p$ decides $\sigma$. We can assume without loss of generality that $p\uhr f(\beta)\Vdash p(f(\beta))\neq 1_{\dot{\dP}_{f(\beta),f(\beta+1)}}$ for all $\beta\in\cf(\rho)$ (because we are taking full support limits at every point in the image of $f$). We define by recursion on $\beta<\cf(\rho)$ a condition $p^*\uhr f(\beta)\in\dP_{f(\beta)}$ such that $p^*\uhr f(\beta)\leq_{f(\beta),0} p\uhr f(\beta)$ and
		$$p^*\uhr f(\beta)\Vdash_{\dP_{f(\beta)}}\neg\sigma_{\beta}$$
		where
		$$\sigma_{\beta}:=\exists q\in\dP_{f(\beta),\rho}(q\dot{\leq}_{(f(\beta),\rho),0}p\uhr[f(\beta),\rho)\wedge q||_{\dot{\dP}_{f(\beta),\rho}}\sigma)$$
		Suppose $p^*\uhr f(\beta)$ has been defined. Let $p^*(f(\beta))$ be such that
		$$p^*\uhr f(\beta)\Vdash \left(p^*(f(\beta))\dot{\leq}_{f(\beta),0}\;p(f(\beta))\text{ and }p^*(f(\beta))||_{\dP_{f(\beta),f(\beta+1)}}\sigma_{f(\beta+1)}\right)$$
		This is possible as $(\dot{\dP}_{f(\beta),f(\beta+1)},\dot{\leq}_{f(\beta),f(\beta+1)},\dot{\leq}_{(f(\beta),f(\beta+1)),0})$ is forced to have the Prikry property. We claim that $p^*\uhr f(\beta+1)\Vdash\neg\sigma_{f(\beta+1)}$, i.e. $p^*\uhr f(\beta)\Vdash p^*(f(\beta))\Vdash\neg\sigma_{f(\beta+1)}$. Otherwise there is $r\leq p^*\uhr f(\beta)$ forcing $p^*(f(\beta))\not\Vdash\neg\sigma_{f(\beta+1)}$, so (because $p^*(f(\beta))$ is forced to decide $\sigma_{f(\beta+1)}$) $r\Vdash p^*(f(\beta))\Vdash\sigma_{f(\beta+1)}$. But then we can take (by the maximum principle) $q\in\dP_{f(\beta+1)}$ which is forced to witness $\dP_{f(\beta+1)}$ and $p^*(f(\beta))^{\frown}q$ is forced by $r\leq p^*\uhr f(\beta)$ to witness $\sigma_{f(\beta)}$, a contradiction.
		
		Suppose $p^*\uhr f(\beta)$ has been defined for all $\beta<\beta'$ and $\beta'$ is a limit. We need to show that $p^*\uhr f(\beta')\Vdash\neg\sigma_{f(\beta')}$. Otherwise there is $r\leq p^*\uhr f(\beta')$ forcing $\sigma_{f(\beta')}$, witnessed by some $\dot{q}\in\dP_{f(\beta'),\rho}$. But then there is $\beta<\beta'$ such that $r\uhr[f(\beta),f(\beta'))\leq_{(f(\beta),f(\beta')),0}p^*\uhr f(\beta')$ (recall that $p\uhr f(\beta)\Vdash p(f(\beta))\neq 1_{\dot{\dP}_{f(\beta),f(\beta+1)}}$ for every $\beta\in\cf(\rho)$), so only finitely many non-direct extensions are allowed). However, then $(r\uhr[f(\beta),f(\beta')))^{\frown}\dot{q}$ is forced by $r\uhr f(\beta)\leq p^*\uhr f(\beta)$ to witness $\sigma_{f(\beta)}$, a contradiction.
		
		Lastly, simply let $r\leq p^*$ decide $\sigma$. Then we obtain a contradiction just like in the limit step before, since $r$ non-directly extends $p^*$ only finitely often.
		
		For the second case, assume that $\rho$ is inaccessible. Let $p\in\dP_{\rho}$ and let $\sigma$ be a statement in the forcing language. Let $\rho'<\rho$ be such that $p\in\dP_{\rho'}$.  Then $p\Vdash_{\dP_{\rho'}}\exists\tau(\tau||_{\dot{\dP}_{\rho',\rho}}\sigma)$. By the maximum principle, fix such a $\tau$. Let $p'\leq_{\rho',0}p$ such that $p'$ decides whether $\tau\Vdash_{\dot{\dP}_{\rho',\rho}}\sigma$. Then, as in the case of the two-step iteration, $(p',\tau)\leq_{\rho,0}p$ and $(p',\tau)$ decides $\sigma$ (here we use that we can take arbitrarily many non-direct extensions outside of the support).
	\end{proof}
	
	Note that if $((\dP_{\alpha},\leq_{\alpha},\leq_{\alpha,0}),(\dot{\dQ}_{\alpha},\dot{\leq}_{\alpha},\dot{\leq}_{\alpha,0}))_{\alpha<\rho}$ is an Easton-support Magidor iteration of Prikry-type forcings of length $\rho$, the sequence of the direct orderings $((\dP_{\alpha},\leq_{\alpha,0}),(\dot{\dQ}_{\alpha},\dot{\leq}_{\alpha,0}))$ is not equal to the Easton-support iteration of the forcings $((\dP_{\alpha},\leq_{\alpha,0}),(\dot{\dQ}_{\alpha},\dot{\leq}_{\alpha,0}))_{\alpha<\rho}$, since we are forcing over $(\dP_{\alpha},\leq_{\alpha})$ and not $(\dP_{\alpha},\leq_{\alpha,0})$. Despite this, since $\leq_{\alpha}$ refines $\leq_{\alpha,0}$ by assumption, the usual proof of the iterability of ${<}\,\mu$-closure works in this context as well and we have:
	
	\begin{mylem}\label{DirectClosureLimit}
		Let $((\dP_{\alpha},\leq_{\alpha},\leq_{\alpha,0}),(\dot{\dQ}_{\alpha},\dot{\leq}_{\alpha},\dot{\leq}_{\alpha,0}))_{\alpha<\rho}$ be an Easton-support Magidor iteration of Prikry-type forcings and $\mu$ a cardinal which is below the first inaccessible. If for all $\alpha<\rho$, $\dP_{\alpha}$ forces $(\dot{\dQ}_{\alpha},\dot{\leq}_{\alpha,0})$ to be ${<}\,\check{\mu}$-closed, then for all $\alpha<\rho$, $(\dP_{\alpha},\leq_{\alpha,0})$ is ${<}\,\mu$-closed.
	\end{mylem}
	
	We note another easy result for later. It is clear that whenever $(\dP,\leq,\leq_0)$ is a Prikry-type forcing and the direct extension ordering $\leq_0$ is ${<}\,\mu$-closed, we can decide names for ordinals ${<}\,\gamma$ using direct extensions whenever $\gamma<\mu$. In most cases (e.g. when $\dP$ is Prikry forcing at a measurable cardinal $\mu$) this is the best we can hope for. However, in our case $\mu$ will often be a successor cardinal. In this case, we have the following:
	
	\begin{mylem}\label{PrikryPumpUp}
		Let $(\dP,\leq,\leq_0)$ be a Prikry-type forcing and $\mu$ a regular cardinal such that $(\dP,\leq_0)$ is ${<}\,\mu$-closed. Then:
		\begin{enumerate}
			\item Whenever $\gamma<\mu$, $p\in\dP$ and $\tau$ is a $\dP$-name with $p\Vdash\tau<\check{\gamma}$, there is $q\leq_0p$ and $\alpha$ such that $q\Vdash\tau=\check{\alpha}$.
			\item Whenever $\gamma,\gamma'<\mu$, $p\in\dP$ and $\tau$ is a $\dP$-name with $p\Vdash\tau\colon\check{\gamma}\to\check{\gamma}'$, there is $q\leq_0p$ and $f\colon\gamma\to\gamma'$ such that $q\Vdash\tau=\check{f}$.
			\item Assume that $\mu$ is not a strong limit. Whenever $p\in\dP$ and $\tau$ is a $\dP$-name with $p\Vdash\tau<\check{\mu}$, there is $q\leq_0p$ and $\alpha$ such that $q\Vdash\tau=\check{\alpha}$.
		\end{enumerate}
	\end{mylem}
	
	\begin{proof}
		Easy. For (3) we can use that ordinals below $\mu$ correspond to functions from $\nu$ to $2$ for some $\nu<\mu$ because $\mu$ is not a strong limit.
	\end{proof}
	
	\section{The Laver-Ideal Property}
	
	The idea of using Namba forcing associated to certain ideals goes as far back as work of Bukovsk\'y and Copl\'akov\'a-Hartov\'a (see \cite{BukovskyCoplakovaMinimal}). Of special interest is the usage of ideals which cause the Namba forcing to behave similarly to Prikry forcing by having a closed dense subset of the collection of positive sets. This idea was recently used by Cox and Krueger (see \cite{CoxKruegerNamba}) in order to obtain stationarily many models which are weakly guessing but not internally approachable.
	
	\begin{mydef}
		Let $\mu\leq\kappa$ be regular cardinals. We say that $\LIP(\mu,\kappa^+)$ holds if there is $I$, a ${<}\,\kappa^+$-complete normal ideal over $\kappa^+$ such that there is $B\subseteq I^+$ which is dense in $I^+$ and ${<}\,\mu$-closed.
	\end{mydef}

	The consistency of this property is attributed to unpublished work of Laver (see \cite[page 291]{GalvinJechMagidorIdealGame}). A published proof of almost the same property was given by Galvin, Jech and Magidor (see \cite[Theorem 4]{GalvinJechMagidorIdealGame}). We will slightly adapt their proof because we want the set $B$ to be ${<}\,\mu$-closed and not merely ${<}\,\mu$-strategically closed. The Laver-Ideal property also occurs in work of Shelah (see \cite[Chapter X, Def. 4.10]{ShelahProperImproper}).

	\begin{mylem}\label{LIPLemma}
		Let $\nu$ be regular and $\kappa>\nu$ measurable. Then $\Coll(\nu,<\kappa)$ forces $\LIP(\nu,\nu^+)$. Moreover, the ideal witnessing $\LIP(\nu,\nu^+)$ concentrates on ordinals of cofinality $\nu$.
	\end{mylem}
	
	\begin{mybem}
		Actually, work of Shelah (see \cite[Claim 9]{ShelahLargeNormal}) shows that any ideal witnessing even $\LIP(\omega_1,\nu^+)$ must concentrate on ordinals of cofinality $\nu$.
	\end{mybem}
	
	\begin{proof}[Proof of Lemma \ref{LIPLemma}]
		Let $G$ be $\Coll(\nu,<\kappa)$-generic. Let $U$ be a ${<}\,\kappa$-complete ultrafilter over $\kappa$ in $V$. In $V[G]$, let $I$ consist of those $X\subseteq\kappa$ such that $X\cap Y=\emptyset$ for some $Y\in U$. Since $U$ concentrates on regular cardinals $>\nu$, $I$ concentrates on ordinals of cofinality $\nu$.
		
		\begin{myclaim}
			$I$ is ${<}\,\kappa$-complete and normal.
		\end{myclaim}
		
		\begin{proof}
			In $V$, let $\dot{f}$ and $p\in G$ be such that $p$ forces $\dot{f}$ to be a function from some $\check{\gamma}$ into $\dot{I}^+$, where $\gamma<\kappa$. For each $\alpha<\gamma$, let $A_{\alpha}$ be a maximal antichain of conditions $q\leq p$ forcing $\dot{f}(\check{\alpha})\cap\check{X}_q^{\alpha}=\emptyset$ for some $X_q^{\alpha}\in U$. By the $\kappa$-cc of the collapse, $|A_{\alpha}|<\kappa$ and so $X_{\alpha}:=\bigcap_{q\in A_{\alpha}}X_q^{\alpha}\in U$. Moreover, $\bigcap_{\alpha<\gamma}X_{\alpha}\in U$. But $p\Vdash\bigcup_{\alpha<\gamma}\dot{f}(\check{\alpha})\cap\check{X}=\emptyset$, so we are done.
			
			Now let $\dot{f}$ be forced by some $p\in G$ to be a function from $\check{\kappa}$ into $\dot{I}^+$. As before, find $(X_{\alpha})_{\alpha<\kappa}$ in $U$ such that $p\Vdash\dot{f}(\check{\alpha})\cap\check{X}_{\alpha}=\emptyset$. Let $X$ be the diagonal intersection of all $X_{\alpha}$. Then $p$ forces that the diagonal union of $\dot{f}$ is disjoint from $\check{X}$.
		\end{proof}
		
		Now we construct $B$ so that $I$ and $B$ together witness $\LIP(\nu,\nu^+)$. Work in $V$. Let $\dQ:=\Coll(\nu,<\kappa)$ and for $\alpha<\kappa$, let $\dQ_{\alpha}:=\{p\in\dQ\;|\;\dom(p)\subseteq\nu\times\alpha\}$, $\dQ^{\alpha}:=\{p\in\dQ\;|\;\dom(p)\cap\nu\times\alpha=\emptyset\}$. Clearly $\dQ$ is isomorphic to $\dQ_{\alpha}\times\dQ^{\alpha}$.
		
		Let $\tau$ be a $\dQ$-name for an element of $\dot{I}^+$ (without loss of generality assume that this is forced by $\emptyset$). Then there is a dense set $D(\tau)$ of conditions $p$ such that $X_{\tau,p}:=\{\alpha<\kappa\;|\;\exists q\in\dQ^{\alpha}\;p\cup q\Vdash\check{\alpha}\in\tau\}$ (which is in $V$) is in $U$ (see \cite[Lemma 2]{GalvinJechMagidorIdealGame}). In $V[G]$, whenever $p\in D(\tau)\cap H$ and $f\in V$ is a function such that for each $\alpha\in X_{\tau,p}$, $f(\alpha)$ witnesses $\alpha\in X_{\tau,p}$, let
		$$A(\tau,p,f,G):=\{\alpha\in X_{\tau,p}\;|\;f(\alpha)\in G\}$$
		Let $B$ be the collection of all $A(\tau,p,f,G)$. Clearly, for all $\tau$ which are forced to be in $\dot{I}^+$, we can find $p$ and $f$ such that $A(\tau,p,f,G)\subseteq\tau^G$, so $B$ is indeed dense in $I^+$. Moreover, by \cite[Lemma 3]{GalvinJechMagidorIdealGame}, each $A(\tau,p,f,G)$ is in $I^+$. We are done after showing:
		
		\begin{myclaim}\label{Claim2LIP}
			Whenever $(A_{\alpha})_{\alpha<\gamma}\in{}^{<\nu}B$ is descending, $\bigcap_{\alpha<\gamma}A_{\alpha}\in I^+$.
		\end{myclaim}
		\begin{proof}
			In $V[G]$, let $(\tau_i)_{i<\gamma}$, $(p_i)_{i<\gamma}$ and $(f_i)_{i<\gamma}$ be sequences such that the sequence $(A(\tau_i,p_i,f_i,G))_{i<\gamma}$ is descending, where $\gamma<\nu$. We can assume that all these sequences are in $V$ by the ${<}\,\nu$-distributivity of $\dQ$. Let $r\in G$ force that the sequence is descending.
			
			\begin{mysclaim}
				For all $i<j<\gamma$ there is $\alpha_{i,j}$ such that $f_i(\alpha)$ and $f_j(\alpha)$ are compatible for all $\alpha\in X_{\tau_i,p_i}\cap X_{\tau_j,p_j}\smallsetminus\alpha_{i,j}$.
			\end{mysclaim}
			\begin{proof}
				We know that $r\Vdash A(\check{\tau}_i,\check{p}_i,\check{f}_i,\Gamma)\supseteq A(\check{\tau}_j,\check{p}_j,\check{f}_j,\Gamma)$. Let $\alpha_{i,j}$ be such that $r':=r\cup p_i\cup p_j\in\dQ_{\alpha_{i,j}}$. Then whenever $\alpha\in (X_{\tau_i,p_i}\cap X_{\tau_j,p_j})\smallsetminus\alpha_{i,j}$, we know that $r'\cup f_j(\alpha)\Vdash\check{\alpha}\in A(\tau_i,\check{p}_i,\check{f}_i,\Gamma)$, so $r'\cup f_j(\alpha)\Vdash\check{f}_i(\check{\alpha})\in\Gamma$ which implies that $f_i(\alpha)$ and $f_j(\alpha)$ are compatible.
			\end{proof}
			
			Let $X:=\bigcap_{i<\gamma}X_{\tau_i,p_i}$ (which is in $U$). In $V$, let $p:=\bigcup_{i<\gamma}p_i$ and let $f(\alpha):=\bigcup_{i<\gamma}f_i(\alpha)$ for $\alpha\in X\smallsetminus\sup_{i,j}\alpha_{i,j}$, empty otherwise. This is possible as $\{p_i\;|\;i<\gamma\}$ as well as $\{f_i(\alpha)\;|\;i<\gamma\}$ for every $\alpha\in X\smallsetminus\sup_{i,j}\alpha_{i,j}$ is a directed subset of $\dQ$ and $\dQ$ is ${<}\,\nu$-directed closed.
			
			\begin{myclaim}
				In $V[G]$, $A:=\bigcap_{i<\gamma}A(\tau_i,p_i,f_i,G)\in I^+$.
			\end{myclaim}
			
			\begin{proof}
				Assume toward a contradiction that there is $r\in G$ and $C\in U$ such that $r\Vdash\dot{A}\cap\check{C}=\emptyset$. Assume $r\leq p$ and let $\alpha$ be such that $r\in\dQ_{\alpha}$. Let $\beta\in X\cap C$ with $\beta>\alpha,\sup_{i,j}\alpha_{i,j}$. Then clearly $r\cup f(\beta)$ is a condition. However, for any $i<\gamma$, $r\cup f(\beta)\leq p_i\cup f_i(\beta)$, so for any $i<\gamma$, $r\cup f(\beta)\Vdash \check{p}_i\cup \check{f}_i(\check{\beta})\in\Gamma$, which implies $r\cup f(\beta)\Vdash\check{\beta}\in A(\check{\tau}_i,\check{p}_i,\check{f}_i,\Gamma)$. In summary, $r\cup f(\beta)\Vdash\check{\beta}\in \dot{A}\cap \check{C}$, a contradiction.
			\end{proof}
			
			This ends the proof of Claim \ref{Claim2LIP}.
		\end{proof}
		So now we can simply take another element of $B$ which is below $\bigcap_{\alpha<\gamma}A_{\alpha}$ and obtain that $B$ is ${<}\,\nu$-closed.
	\end{proof}

	The statement we are really after is the following:
	
	\begin{mylem}\label{OmegaLIPLemma}
		Let $(\kappa_n)_{n\in\omega}$ be an increasing sequence of supercompact cardinals. Let $\mu<\kappa_0$ be regular. There is a ${<}\,\mu$-directed closed forcing extension in which for every $n\in\omega$, $\kappa_n=\mu^{+n+1}$ and $\LIP(\mu,\kappa_n)$ holds.
	\end{mylem}
	
	\begin{proof}
		By a technique of Laver (see \cite{LaverIndestruct}) we can assume that the supercompactness of every $\kappa_n$ is indestructible under ${<}\,\kappa_n$-directed closed forcing (in particular, this preparation is possible with a ${<}\,\mu$-directed closed forcing). Let $\dP:=\prod_{n\in\omega}\Coll(\kappa_{n-1},<\!\kappa_n)$ (with $\kappa_{-1}:=\mu$). Clearly, for every $n\in\omega$, $\dP$ forces $\check{\kappa}_n=\check{\mu}^{+n+1}$.
		
		Let $n\in\omega$ be arbitrary. Then we can write
		$$\dP=\prod_{k<n}\Coll(\kappa_{k-1},<\!\kappa_k)\times\Coll(\kappa_{n-1},<\!\kappa_n)\times\prod_{k>n}\Coll(\kappa_{k-1},<\!\kappa_k)$$
		After forcing with $\prod_{k>n}\Coll(\kappa_{k-1},<\!\kappa_k)$, $\kappa_n$ is still supercompact and in particular measurable. Ergo $\Coll(\kappa_{n-1},<\!\kappa_n)\times\prod_{k>n}\Coll(\kappa_{k-1},<\!\kappa_k)$ forces $\LIP(\check{\kappa}_{n-1},\check{\kappa}_n)$. We are done after showing the following:
		
		\begin{myclaim}
			Let $\dP$ be of size ${<}\,\kappa$ and ${<}\,\delta$-closed. If $\LIP(\delta,\kappa)$ holds, it is preserved by $\dP$.
		\end{myclaim}
		
		\begin{proof}
			Let $I$ and $B$ witness $\LIP(\delta,\kappa)$ and let $G$ be $\dP$-generic. In $V[G]$, let $J$ be generated by $I$, i.e. $X\in J$ if and only if $X\subseteq Y$ for some $Y\in I$. We will verify that $J$ and $B$ witness $\LIP(\delta,\kappa)$. Clearly, $B$ is still ${<}\,\delta$-closed by the closure of $\dP$ and $J$ is normal and ${<}\,\kappa$-complete by the $\kappa$-cc of $\dP$.
			
			Let $\tau$ be a $\dP$-name and $p\in\dP$ such that $p\Vdash\tau\in\dot{J}^+$. Then $p$ forces that $\tau$ is contained in $\bigcup_{q\leq p}\{\alpha\in\kappa\;|\;q\Vdash\check{\alpha}\in\tau\}$. As $I$ is ${<}\,\kappa$-complete and $|\dP|<\kappa$, there is $q\leq p$ such that $\{\alpha\in\kappa\;|\;q\Vdash\check{\alpha}\in\tau\}\in I^+$. Ergo there is $X\in B$ with $X\subseteq\{\alpha\in\kappa\;|\;q\Vdash\check{\alpha}\in\tau\}$. It follows that $q\Vdash\check{X}\subseteq\tau$.
		\end{proof}
		
		Since $\prod_{k<n}\Coll(\kappa_{k-1},<\!\kappa_k)$ is ${<}\,\mu$-closed and of size $\kappa_{n-1}<\kappa_n$, we are done.
	\end{proof}
	
	\section{The Namba Forcing}
	
	Namba forcing is intimately connected to the study of good points. This is mostly due to the fact that (at least in the Laver-style variant with splitting into normal ideals) it adds exact upper bounds to scales and thus, by choosing the splitting sets correctly, we can decide whether we want points to be good or bad. However, in this work we will use Namba forcing for a different reason: By techniques of Shelah (see \cite[Chapter XI]{ShelahProperImproper}), assuming $\GCH$, Namba forcing does not add any new functions from $\omega$ into any cardinal smaller than the limit superior of the completeness of the used ideals. Despite this, Namba forcing and any iteration of Namba forcing together with a sufficiently closed forcing will have a weak variant of the approximation property (see \cite[Theorem 2.1]{LevineMildenbergerGoodScaleNonCompactSquares}). The combination of these two properties is what allows us to show that any such iteration will not make $\aleph_{\omega+1}$ $d$-approachable with respect to any normal subadditive coloring $d$ on $\aleph_{\omega+1}$.
	
	The basic poset could just as well be defined with ``normal'' Namba forcing (just splitting into the nonstationary ideal) on $\prod_n\aleph_n$ and still have the property that it does not make $\aleph_{\omega+1}$ $d$-approachable for any normal subadditive coloring. However, we will use a variant of Namba forcing which splits into $\LIP$-ideals for the following two reasons: For one, even in the case where we collapse $\aleph_{\omega+1}$ to $\aleph_1$ it seems necessary to obtain the desired regularity properties for the tail forcing and secondly, in the case where we collapse $\aleph_{\omega+1}$ to larger $\aleph_n$ we need it to be able to show that all cardinals below and including $\aleph_n$ are preserved.
	
	\begin{mydef}
		Let $(\kappa_n)_{n\in\omega}$ be an increasing sequence of regular cardinals. For each $n\in\omega$, let $I_n$ be a ${<}\,\kappa_n$-complete ideal over $\kappa_n$. Then $\dL((\kappa_n,I_n)_{n\in\omega})$ consists of all $p\subseteq\bigcup_{k\in\omega}\prod_{n\leq k}\kappa_n$ such that
		\begin{enumerate}
			\item $p$ is closed under restriction.
			\item there is some $\stem(p)\in p$ such that for all $s\in p$, $s\subseteq\stem(p)$ or $\stem(p)\subseteq s$ and whenever $\stem(p)\subseteq s$,
			$$\osucc_p(s):=\{\alpha\in\kappa_{|s|}\;|\;s^{\frown}\alpha\in p\}\in I_{|s|}^+$$
		\end{enumerate}
		We let $q\leq p$ if and only if $q\subseteq p$. We let $q\leq_0p$ if and only if $q\leq p$ and $\stem(q)=\stem(p)$.
	\end{mydef}

	This forcing was used by Cox and Krueger (see \cite{CoxKruegerNamba}) to obtain non-internally unbounded weakly guessing models. In that same work, they showed that if the $I_n$'s are $\LIP$-ideals, the associated Namba forcing has a weak form of the approximation property. However, it is not known to us whether their proof can also work for iterations of Namba forcing and closed forcing. This is due to the fact that they seem to require a stronger closure than we have: In order to obtain the weak ${<}\,\mu$-approximation property they require a ${<}\,\mu^+$-closed direct extension ordering while our direct extension ordering is merely ${<}\,\mu$-closed.
	
	For the rest of this section, we fix the following:
	
	\begin{myass}
		 $(\kappa_n)_{n\in\omega}$ is an increasing sequence of regular cardinals and for each $n\in\omega$, $I_n$ is a ${<}\,\kappa_n$-complete ideal over $\kappa_n$. Define $\kappa^*:=(\sup_n\kappa_n)^+$. $\mu<\kappa_0$ is a regular cardinal. For each $n\in\omega$ there is a set $B_n\subseteq I_n^+$ which is dense in $I_n^+$ and ${<}\,\mu$-closed.
	\end{myass}
	
	In other words, we are assuming that each $I_n$ witnesses $\LIP(\mu,\kappa_n)$.
	
	The following is \cite[Lemma 6.5]{CoxKruegerNamba}:
	
	\begin{myfact}\label{PrikryProperty}
		Whenever $p$ is a condition in $\dL((\kappa_n,I_n)_{n\in\omega})$ and $\tau$ is an $\dL((\kappa_n,I_n)_{n\in\omega})$-name for an ordinal such that $p\Vdash\tau<\check{\lambda}$ for some ordinal $\lambda<\kappa_{|\stem(p)|}$, there is $q\leq_0p$ which forces $\tau=\check{\delta}$ for some $\delta<\lambda$.
	\end{myfact}
	
	Since we can code statements in the forcing language as ordinals in $\{0,1\}$, it follows that $\dL((\kappa_n,I_n)_{n\in\omega})$ has the Prikry property and is thus a Prikry-type forcing.
	
	For names of ordinals in $\kappa^*$, we of course cannot hope for decisions using direct extensions. However, we do have the following:
	
	The following first appeared in work of the second author and Mildenberger (see \cite{LevineMildenbergerGoodScaleNonCompactSquares}). Since our forcing notion is slightly different and we are slightly generalizing the lemma (in order to later be able to collapse $\aleph_{\omega+1}$ to any $\aleph_n$ in a non-fresh way), we are giving the proof here, although the argument is quite similar to the one in that article.
	
	\begin{mysen}\label{ApproxProp}
		Let $\dot{\dQ}$ be an $\dL((\kappa_n,I_n)_{n\in\omega})$-name for a ${<}\,\mu$-closed forcing. Assume $\sup_n\kappa_n$ is a strong limit and $2^{\sup_n\kappa_n}=(\sup_n\kappa_n)^+$.
		
		Then whenever $\dot{F}$ is an $\dL((\kappa_n,I_n)_{n\in\omega})*\dot{\dQ}$-name for a cofinal function from $\mu$ into $\kappa^*$, it is forced that $\dot{F}\uhr \check{i}\notin V$ for some $i\in\mu$.
	\end{mysen}
	
	Let us note that whenever a poset $\dQ$ does not have an explicitely defined direct extension ordering, we simply let $\leq_0=\leq$ (it follows easily that $(\dQ,\leq,\leq_0)$ then has the Prikry property). Before we can prove Theorem \ref{ApproxProp} we will need some preliminary results.
	
	\begin{mypro}\label{EventualDecision}
		Let $\dot{\dQ}$ be an $\dL((\kappa_n,I_n)_{n\in\omega})$-name for a forcing poset. Suppose $\tau$ is an $\dL((\kappa_n,I_n)_{n\in\omega})*\dot{\dQ}$-name and $(p,\dot{c})\in\dL((\kappa_n,I_n)_{n\in\omega})*\dot{\dQ}$ forces $\tau\in V$. Then there is $(q,\dot{d})\leq_0(p,\dot{c})$ and some $k$ such that whenever $t\in q$, $|t|=k$, $(q\uhr t,\dot{d})\Vdash\tau=\check{x}$ for some $x$.
	\end{mypro}
	
	\begin{proof}
		It is clear that the set $D$ of all $r\in\dL((\kappa_n,I_n)_{n\in\omega})$ such that for some $\dot{e}$ and $x$, $(r,\dot{e})\leq(1_{\dL},\dot{d})$ and forces $\tau=\check{x}$, is open dense in $\dL((\kappa_n,I_n)_{n\in\omega})$. By standard statements about Laver-style Namba forcings (see \cite[Lemma 6.4]{CoxKruegerNamba}), there is $q\leq_0p$ and $k$ such that whenever $t\in q$, $|t|=k$, $q\uhr t\in D$. For any $t\in q$ with $|t|=k$, let $\dot{e}_t$ witness $q\uhr t\in D$. Since $\{q\uhr t\;|\;t\in q, |t|=k\}$ is an antichain, we can find $\dot{d}$ such that for all $t\in q$ with $|t|=k$, $q\uhr t\Vdash\dot{d}=\dot{e}_t$. It follows that $(q,\dot{d})$ is as required.
	\end{proof}
	
	We will also need the concept of a fusion sequence:
	
	\begin{mydef}
		Let $p,q\in\dL((\kappa_n,I_n)_{n\in\omega})$. For $n\in\omega$, we write $q\leq_np$ if $q\leq_0 p$ and for all $t\in q$, if $|t|\leq |\stem(p)|+n$, then $t\in p$.
		
		We say that $(p_n)_{n\in\omega}$ is a \emph{fusion sequence} if $p_{n+1}\leq_np_n$ for all $n\in\omega$.
	\end{mydef}
	
	The following is standard:
	
	\begin{myfact}
		If $(p_n)_{n\in\omega}$ is a fusion sequence of conditions in $\dL((\kappa_n,I_n)_{n\in\omega})$, then $\bigcap_{n\in\omega}p_n\in\dL((\kappa_n,I_n)_{n\in\omega})$ and extends every $p_n$.
	\end{myfact}
	
	The reason for introducing the Laver-ideal property is that by using these ideals we can show that our version of Namba forcing has a sufficiently closed direct extension ordering:
	
	\begin{mydef}\label{DefSuborder}
		Let $\dL'((\kappa_n,I_n)_{n\in\omega})$ consist of those $p\in\dL((\kappa_n,I_n)_{n\in\omega})$ such that whenever $t\in p$ and $\stem(p)\subseteq t$, $\osucc_p(t)\in B_{|t|}$.
	\end{mydef}
	
	\begin{mylem}\label{SuborderClosed}
		The suborder $\dL'((\kappa_n,I_n)_{n\in\omega})$ is dense in $\dL((\kappa_n,I_n)_{n\in\omega})$ with respect to both $\leq$ and $\leq_0$ and the direct extension ordering $\leq_0$ is ${<}\,\mu$-closed on $\dL'((\kappa_n,I_n)_{n\in\omega})$.
	\end{mylem}
	
	\begin{proof}
		It is easily seen using a simple fusion argument and the density of each $B_n$ that $\dL'((\kappa_n,I_n)_{n\in\omega})$ is dense in $\dL((\kappa_n,I_n)_{n\in\omega})$ (for a proof, see \cite[Lemma 6.12]{CoxKruegerNamba}). Moreover, the set is clearly ${<}\,\mu$-closed because we can take the intersection of a descending sequence of conditions, using the ${<}\,\mu$-closure of each $B_n$ (see the proof of \cite[Lemma 6.12]{CoxKruegerNamba}).
	\end{proof}
	
	From now on, we will tacitly replace $\dL((\kappa_n,I_n)_{n\in\omega})$ by $\dL'((\kappa_n,I_n)_{n\in\omega})$ in order to simplify notation.
	
	Lastly we show that our iteration does not collapse $\kappa^*$ ``too much'':
	
	\begin{mylem}\label{KappaStarNotToMu}
		Let $\mu<\kappa_0$. Let $\dot{\dQ}$ be an $\dL((\kappa_n,I_n)_{n\in\omega})$-name for a ${<}\,\mu$-closed poset. Let $\kappa^*:=(\sup_n\kappa_n)^+$. After forcing with $\dL((\kappa_n,I_n)_{n\in\omega})*\dot{\dQ}$, $\cf(\kappa^*)\geq\mu$.
	\end{mylem}
	
	\begin{proof}
		This is classical for variants of Namba forcing. If $\tau$ is an $\dL((\kappa_n,I_n)_{n\in\omega})$-name for an ordinal below $\kappa^*$, forced by $p$, there is a ${<}\,\kappa^*$-sized set $x$ and $p'\leq_0p$ such that $p\Vdash\tau\in\check{x}$: Let $D$ be the open dense set of all $q\in\dL((\kappa_n,I_n)_{n\in\omega})$ which force $\tau=\check{\alpha}_q$ for some $\alpha_q$. As in the proof of Lemma \ref{EventualDecision}, there is $p'\leq_0p$ and $k$ such that whenever $t\in p'$, $|t|=k$, $p'\uhr t\in D$, witnessed by some $\alpha_t$. It follows that $p'\Vdash\tau\in\{\alpha_t\;|\;t\in p'\wedge|t|=k\}$ and the latter set has size ${<}\,\kappa^*$.
		
		The previous statement combined with the ${<}\,\mu$-closure of a dense subset of the direct extension ordering on $\dL((\kappa_n,I_n)_{n\in\omega})$ (see Lemma \ref{SuborderClosed}) shows that no unbounded function from any $\gamma<\mu$ into $\kappa^*$ is added. This is forced to be preserved by $\dot{\dQ}$ by its closure.
	\end{proof}

	We can now prove that $\dL((\kappa_n,I_n)_{n\in\omega})$ and any iteration of it with a ${<}\,\mu$-closed forcing notion has the stated weak form of the ${<}\,\mu$-approximation property.

	\begin{proof}[Proof of Theorem \ref{ApproxProp}]
		For simplicity, define $\dL:=\dL((\kappa_n,I_n)_{n\in\omega})$ and assume $\dot{F}$ is as in the statement of the theorem. Suppose toward a contradiction that there is a condition $(p^*,\dot{q}^*)\in\dL*\dot{\dQ}$ which forces ``$\forall i<\mu(\dot{F}\uhr i\in V)$''. Assume for notational simplicity that $(p^*,\dot{q}^*)=1_{\dL*\dot{\dQ}}$.
		
		For $i<\mu$ and $(q,\dot{d})\in\dL*\dot{\dQ}$ let $\phi(i,q,\dot{d})$ state that there exists a sequence $(A_{\alpha})_{\alpha\in\osucc_q(\stem(q))}$ such that for all $\alpha,\beta\in\osucc_q(\stem(q))$, $|A_{\alpha}|<\kappa^*$, $A_{\alpha}$ and $A_{\beta}$ are disjoint for $\alpha\neq\beta$ and $(q\uhr(\stem(q)^{\frown}\alpha),\dot{d})\Vdash\dot{F}\uhr\check{i}\in\check{A}_{\alpha}$.
		
		\setcounter{myclaim}{0}
		
		\begin{myclaim}\label{Claim1}
			If $j<\mu$ and $(p,\dot{c})\in\dL*\dot{\dQ}$, there is some $i\in(j,\mu)$ and $(q,\dot{d})\leq_0(p,\dot{c})$ such that $\phi(i,q,\dot{d})$ holds.
		\end{myclaim}
		\begin{proof}
			Let $W:=\osucc_p(\stem(p))$. By induction on $\alpha\in W$ we will define a sequence of conditions $(q_{\alpha},\dot{d}_{\alpha})_{\alpha\in W}$ with $(q_{\alpha},\dot{d}_{\alpha})\leq_0(q\uhr(\stem(q)^{\frown}\alpha),\dot{c})$, a sequence of natural numbers $(n_{\alpha})_{\alpha\in W}$, a sequence of ordinals $(i_{\alpha})_{\alpha\in W}$ below $\mu$, and sets $(A_{\alpha})_{\alpha\in W}$ of cardinality ${<}\,\kappa^*$.
			
			Let $\alpha=\min(W)$. Choose any $i_{\alpha}\in(j,\mu)$. Apply Proposition \ref{EventualDecision} to find $(q_{\alpha},\dot{d}_{\alpha})\leq_0(p\uhr(\stem(p)^{\frown}\alpha),\dot{c})$ and $n_{\alpha}\in\omega$ such that whenever $s\in q_{\alpha}$, $|s|=n_{\alpha}$, $(q_{\alpha}\uhr s,\dot{d}_{\alpha})$ decides $\dot{F}\uhr i_{\alpha}$. Let
			$$A_{\alpha}:=\{a\;|\;\exists t\in q_{\alpha}(|t|=n_{\alpha}\wedge (q_{\alpha}\uhr t,\dot{d}_{\alpha})\Vdash\dot{F}\uhr i_{\alpha}=\check{a})\}$$
			Clearly, $|A_{\alpha}|<\kappa^*$.
			
			Now suppose all objects have been defined for $\beta\in W\cap\alpha$ and $\alpha\in W$. Let $B:=\bigcup_{\beta\in W\cap\alpha}A_{\beta}$. Then $|B|<\kappa^*$. So in particular,
			$$(p\uhr (\stem(p)^{\frown}\alpha),\dot{c})\Vdash\{\dot{F}\uhr i\;|\;i\in(j,\mu)\}\not\subseteq \check{B}$$
			since every $\dot{F}\uhr i$ is forced to be bounded in $\kappa^*$ (by Lemma \ref{KappaStarNotToMu}), the negation of the above statement would imply that $\dot{F}$ is forced by some condition to be bounded in $\kappa^*$, a contradiction.
			
			Therefore $(p\uhr (\stem(p)^{\frown}\alpha),\dot{c})\Vdash\exists i\in(j,\mu)(\dot{F}\uhr i\notin \check{B})$. By Fact \ref{PrikryProperty} there exists $(q_{\alpha}',\dot{d}_{\alpha}')\leq_0(p\uhr (\stem(p)^{\frown}\alpha),\dot{c})$ deciding $\dot{i}=\check{i}_{\alpha}$ for some $i_{\alpha}<\mu$. So we have
			$$(q_{\alpha}',\dot{d}_{\alpha}')\Vdash\dot{F}\uhr\check{i}_{\alpha}\notin\check{B}$$
			Now simply proceed as in the base case and find $(q_{\alpha},\dot{d}_{\alpha})\leq_0(q_{\alpha}',\dot{d}_{\alpha}')$ and $n_{\alpha}$ such that whenever $t\in q_{\alpha}$ has length $n_{\alpha}$, $(q_{\alpha}\uhr t, \dot{d}_{\alpha})$ decides $\dot{F}\uhr\check{i}_{\alpha}$ and let $A_{\alpha}$ be the set of possible decisions. Then clearly $|A_{\alpha}|<\kappa^*$, $A_{\alpha}\cap B=\emptyset$ and thus $A_{\alpha}\cap A_{\beta}=\emptyset$ for all $\beta<\alpha$.
			
			Finally we can choose $W'\subseteq W$ which is $I_{|\stem(p)|}$-positive such that there exists $i<\mu$ such that for all $\alpha\in W'$, $i_{\alpha}=i$. Let $q:=\bigcup_{\alpha\in W'}q_{\alpha}$ which is clearly a direct extension of $p$ and let $\dot{d}$ be such that $q_{\alpha}\Vdash\dot{d}=\dot{d}_{\alpha}$ for every $\alpha\in W'$. Clearly $(q,\dot{d})$ is as required.
		\end{proof}
		
		We now define a game $\mathcal{G}_k$ for every $k<\mu$. The game is defined as follows: Player I starts by playing a pair $(Z_0,\delta_0)$ where $Z_0\subseteq\kappa_0$, $Z_0\in I_0$ and $\delta_0<k$. After $(Z_n,\delta_n)$ has been played, Player II chooses $\alpha\in\osucc_{q_{n-1}}(\stem(q_{n-1}))\smallsetminus Z_n$ and $(q_n,\dot{d}_n)\leq_0(q_{n-1}\uhr(\stem(q_{n-1})^{\frown}\alpha),\dot{d}_{n-1})$ and some $i_n\in(\delta_n,k)$ such that $\phi(q_n,\dot{d}_n,i_n)$ holds (we let $(q_{-1},\dot{d}_{-1}):=1_{\dL*\dot{\dQ}}$). Then Player I responds by playing $(Z_n,\delta_n)$ such that $Z_n\subseteq\osucc_{q_n}(\stem(q_n))$, $Z_n\in I_n$ and $\delta_n<k$ (note we are not requiring $\delta_n>i_n$). Player II loses if they cannot play at some stage $n$. Otherwise, they win.
		
		\begin{myclaim}
			There is $k\in\mu\cap\cof(\omega)$ such that Player II has a winning strategy in $\mathcal{G}_k$.
		\end{myclaim}
		
		\begin{proof}
			Assume otherwise. By the Gale-Stewart theorem and since the game is open, Player I has a winning strategy $\sigma_i$ for every $i\in\mu\cap\cof(\omega)$. Let $M\prec H(\Theta)$ be of size ${<}\,\mu$ such that $M$ contains everything relevant (in particular, $(\sigma_i)_{i\in\mu\cap\cof(\omega)}\in M$) and $k:=M\cap\mu\in\mu\cap\cof(\omega)$. We will construct a run of the game in which Player I loses despite using $\sigma_k$. Let $(Z_0,\delta_0)$ be the opening move chosen by $\sigma_k$. At stage $n$, let $W$ be the set of indices $i\in\mu\cap\cof(\omega)$ such that $(q_0,\dot{d}_0,i_0),\dots,(q_{n-1},\dot{d}_{n-1},i_{n-1})$ is a legal sequence of Player II's moves in a game where Player I follows the strategy $\sigma_i$. $W$ is nonempty (as witnessed by $k$) and is in $M$ by elementarity. Let
			$$Y:=\bigcup\{Z_i\;|\;i\in W,\sigma_i((q_0,\dot{d}_0,i_0),\dots,(q_{n-1},\dot{d}_{n-1},i_{n-1}))=(Z_i,\delta_i)\}$$
			Note that $Y\in M$. Furthermore, since $|W|\leq\mu$, $Y\in I_n$. So let $\alpha$ be any element of $\osucc_{q_{n-1}}(\stem(q_{n-1}))\smallsetminus Y$ and apply Claim \ref{Claim1} to $(q_{n-1}\uhr\stem(q_{n-1})^{\frown}\alpha,\dot{d}_{n-1},\delta_n)$ to obtain $(q_n,\dot{d}_n,i_n)$. By elementarity, we can assume $(q_n,\dot{d}_n,i_n)\in M$, so $i_n<k$ and the move is allowed.
		\end{proof}
		
		Now we build a fusion sequence in such a way that any stronger condition deciding $\dot{F}\uhr\check{k}$ will also decide the generic sequence for $\dL$, thereby obtaining a contradiction.
		
		Let $(\delta_n)_{n\in\omega}$ be a sequence converging to $k$. By induction on $n\in\omega$ we will define a fusion sequence $(p_n)_{n\in\omega}$ and a sequence $(\dot{c}_n)_{n\in\omega}$ such that $p_{n+1}\Vdash\dot{c}_{n+1}\leq\dot{c}_n$ and such that:
		\begin{myass}
			For all $n<\omega$, for all $t\in p_n$ with $|t|=n$, there is a sequence $Z_0^t,\dots,Z_n^t$ such that
			$$(Z_0^t,\delta_0),(p_0\uhr(t\uhr 1),\dot{c}_0,i_0),(Z_1^t,\delta_1),(p_0\uhr(t\uhr 2)),\dots,(Z_n^t,\delta_n),(p_n\uhr t,\dot{c}_n,i_n)$$
			is a run of the game $\mathcal{G}_k$ where Player II played according to their winning strategy.
		\end{myass}
		For convenience, define $(p_{-1},\dot{c}_{-1}):=1_{\dL*\dot{\dQ}}$. Assume we have defined $(p_{n-1},\dot{c}_{n-1})$. Let $t\in p_{n-1}$ with $|t|=n$ and let $S_t$ be the set of all $\alpha\in\osucc_{p_{n-1}}(t)$ such that there is some $Z_n^{\alpha}$ such that the winning strategy for Player II applied to
		\begin{align*}
			(Z_0^t,\delta_0),(p_0\uhr(t\uhr 1),\dot{c}_0,i_0),(Z_1^t,\delta_1), \\
			(p_0\uhr(t\uhr 2)),\dots,(Z_n^t,\delta_n),(p_n\uhr t,\dot{c}_n,i_n),(Z_n^{\alpha},\delta_n)
		\end{align*}
		produces $(q_n,\dot{d}_n,i_n)$ where $(q_n,\dot{d}_n)\leq_0(p_{n-1}\uhr t^{\frown}\alpha,\dot{c}_{n-1})$. We claim that $S_t\in I_n^+$. Otherwise $(S_t,\delta_n)$ would be a legal move for player I at stage $n$ with no possible response by player II. So for each $t\in p_{n-1}$ with $|t|=n$ and $\alpha\in S_t$, let $(q_{t,\alpha},\dot{d}_{t,\alpha},i_{t,\alpha})$ be the response of player II according to their winning strategy to $(Z_n^{\alpha},\delta_n)$. Let $p_n:=\bigcup\{q_{t,\alpha}\;|\;t\in p_{n-1},|t|=n,\alpha\in S_t\}$ and let $\dot{d}_n$ be such that $q_{t,\alpha}\Vdash\dot{d}_n=\dot{d}_{t,\alpha}$ for each $t\in p_{n-1}$, $|t|=n$ and $\alpha\in S_t$.
		
		Now let $q:=\bigcap_{n\in\omega}p_n$ be the fusion limit of $(p_n)_{n\in\omega}$ and let $\dot{d}$ be an $\dL$-name for a lower bound of $(\dot{d}_n)_{n\in\omega}$, forced by $q$. We claim that no extension of $(q,\dot{d})$ decides $\dot{F}\uhr\check{k}$. Assume toward a contradiction that $(r,\dot{e})\leq(q,\dot{d})$ forces $\dot{F}\uhr\check{k}=\check{g}$ for some $g\in V$.
		
		Let $b$ be a function on $\omega$ defined by induction as follows such that $b\uhr n\in r$ for all $n\in\omega$. Let $b\uhr|\stem(r)|=r$. If $b\uhr n$ has been defined, then we know by construction that $\phi(i_t,r\uhr(b\uhr n),\dot{d})$ holds and thus there is exactly one $\alpha\in\osucc_{r\uhr(b\uhr n)}(b\uhr n)$ such that $g\uhr i_t\in A_{\alpha}$. Let $b(n):=\alpha$.
		
		Let $H*G$ be $\dL*\dot{\dQ}$-generic containing $(r,\dot{e})$. Then:
		
		\begin{myclaim}
			For every $s\in H$, $\stem(s)=b\uhr|\stem(s)|$.
		\end{myclaim}
		
		\begin{proof}
			We do the proof by induction on $|\stem(s)|$. Let $s\in H$ and assume the statement holds for all $t\in H$ with $|\stem(t)|<|\stem(s)|$. We may assume that $|\stem(s)|>|\stem(r)|$ since otherwise we know $\stem(s)\subseteq\stem(r)=b\uhr|\stem(r)|$ by construction.
			
			Find $t\in H$ such that $\stem(s)=\stem(t)^{\frown}\alpha$. We want to show $b(|\stem(t)|)=\alpha$. Since $r\in H$ and $H$ is a filter, $r$ and $t$ are compatible and in particular there is $t'\leq s,r$ with $\stem(t')=\stem(t)$ (just intersect $r$ and $s$), so we may as well assume $t\leq r$. By the same argument, assume $s\leq t$. Then $\phi(i_{|\stem(t)|},t,\dot{d})$ holds and since $\stem(s)=\stem(t)^{\frown}\alpha$, $(s,\dot{d})\Vdash \dot{F}\uhr\check{i}_{|\stem(t)|}\in\check{A}_{\alpha}$. Since $A_{\alpha}\cap A_{\alpha'}=\emptyset$ for $\alpha'\neq\alpha$ and $(r,\dot{e})$ forces $\dot{F}\uhr\check{i}_{|\stem(s)|}=\check{g}\uhr\check{i}_{|\stem(s)|}\in\check{A}_{b(n)}$, we are done.
		\end{proof}
		
		This is a clear contradiction to the fact that whenever $H$ is $\dL$-generic, the generic function $\bigcup\{\stem(p)\;|\;p\in H\}$ is outside of the ground model, since the latter set is equal to $b$.	Ergo $(q,\dot{d})$ forces $\dot{F}\uhr\check{k}\notin V$, which contradicts our assumptions.
	\end{proof}

	We now present an interesting contrasting property of $\dL((\kappa_n,I_n)_{n\in\omega})$:
	
	\begin{mylem}\label{NoNewFunctions}
		$\dL((\kappa_n,I_n)_{n\in\omega})$ does not add any new functions from $\gamma$ to $\kappa_k$ whenever $\gamma<\mu$ and $k\in\omega$.
	\end{mylem}

	\begin{proof}
		Let $\dot{f}$ be a $\dL((\kappa_n,I_n)_{n\in\omega})$-name for a function from $\check{\gamma}$ to $\check{\kappa}_k$, where $k\in\omega$, forced by some $p\in\dL((\kappa_n,I_n)_{n\in\omega})$. Assume that $|\stem(p)|\geq k+1$. By induction on $\delta<\gamma$, let $p_{\delta}$ be a lower bound (with respect to $\leq_0$) of $(p_{\alpha})_{\alpha<\delta}$ (with $p_0\leq_0p$) which decides $\dot{f}(\check{\delta})$ (this is possible by Lemma \ref{PrikryProperty} and Lemma \ref{SuborderClosed}). Now let $p^*$ be a lower bound of $(p_{\delta})_{\delta<\gamma}$. It follows that $p^*$ decides $\dot{f}$.
	\end{proof}

	The previous two lemmas seem to be in contention: On one hand, every new function on $\mu$ has an initial segment which is in the ground model. On the other hand we are not adding new functions from ordinals below $\mu$ into any $\kappa_n$. However, this tension is precisely what allows us to show that the iteration $\dL((\kappa_n,I_n)_{n\in\omega})*\dot{\Coll}(\check{\mu},\check{\kappa}^*)$ collapses $\kappa^*$ to $\mu$ without making it approachable.
	
	The following is the crux of this paper:
	
	\begin{mysen}\label{NoNewApproach}
		Let $d\colon[\kappa^*]^2\to\omega$ be a $(\kappa_n)_{n\in\omega}$-normal coloring. The forcing $\dL((\kappa_n,I_n)_{n\in\omega})*\dot{\Coll}(\check{\mu},\check{\kappa}^*)$ does not add a cofinal subset of $\kappa^*$ on which $d$ is bounded.
	\end{mysen}

	\begin{proof}
		It follows that after forcing with $\dL*\dot{\Coll}(\check{\mu},\check{\kappa}^*)$, $\cf(\kappa^*)=\mu$. Let $G*H$ be $\dL*\dot{\Coll}(\check{\mu},\check{\kappa^*})$-generic. Assume that in $V[G*H]$ there is an increasing and cofinal function $F\colon\mu\to\kappa^*$ and $n\in\omega$ such that $d(\alpha,\beta)\leq n$ for every $\alpha,\beta\in\im(F)$.
		
		By Lemma \ref{ApproxProp} there is $i<\mu$ such that $F\uhr i\notin V$. We will show that this leads to a contradiction. Let $\alpha:=F(i)$. Then $F\uhr i\colon i\to\{\beta<\alpha\;|\;d(\beta,\alpha)\leq n\}$ since $F$ is increasing and $d$ is bounded on the image of $F$. But $\{\beta<\alpha\;|\;d(\beta,\alpha)\leq n\}$ is in $V$ and has size $\leq\kappa_n$ there (since $d$ is $(\kappa_n)_{n\in\omega}$-normal). So modulo coding in the ground model, $F\uhr i$ is a function from $i$ to $\kappa_n$ which is in $V[G*H]\smallsetminus V$. This contradicts Lemma \ref{NoNewFunctions}, since clearly $\dot{\Coll}(\check{\mu},\check{\kappa}^*)$ is forced not to add any new functions from $i<\mu$ into the ordinals by its closure.
	\end{proof}
	
	The following lemma will be used to show that even in the case $\mu=\aleph_1$, our Namba forcing makes $\kappa^*$ into a good point. Since our Namba forcing is a ``Laver-style'' version (i.e. it splits everywhere above the stem) and we are using normal ideals, its generic is an exact upper bound of any $(\kappa^*,(\kappa_n)_{n\in\omega})$-scale:
	
	\begin{mylem}\label{ExactUpperBound}
		Let $(f_{\alpha})_{\alpha<\kappa^*}$ be a $(\kappa^*,(\kappa_n)_{n\in\omega})$-scale. Let $G$ be $\dL((\kappa_n,I_n)_{n\in\omega})$-generic. In $V[G]$, the set $b:=\bigcup\{\stem(p)\;|\;p\in G\}$ is an exact upper bound of $(f_{\alpha})_{\alpha<\kappa^*}$.
	\end{mylem}
	
	\begin{proof}
		We repeat a proof by Cummings and Magidor (see \cite[Claim 4]{CumMagMMWeakSquare}). It is clear by genericity that $b$ is an upper bound of $(f_{\alpha})_{\alpha<\kappa^*}$: Given any $p\in G$ and $\alpha<\kappa^*$, there is a direct extension $p_0\leq_0p$ in $G$ such that whenever $s\in p_0$ and $n\geq|\stem(p_0)|$, $s(n)> f_{\alpha}(n)$ (if it is defined) by simply taking away the values below $f_{\alpha}$ from the possible ordinal successors.
		
		The more substantial claim is showing that $b$ is exact. To this end, let $\dot{g}$ be an $\dL((\kappa_n,I_n)_{n\in\omega})$-name for an element of $\prod_{n\in\omega}\check{\kappa}_n$ which is forced by some condition $p$ to be ${<^*}\,\dot{b}$. Assume without loss of generality that $\dot{g}$ is forced to be strictly below $\dot{b}$. Let $n:=|\stem(p)|$. If $q\leq p$, $l:=|\stem(q)|$, then for any $\alpha\in\osucc_{\stem(q)}(q)$ there is some $q_{\alpha}\leq_0 q\uhr(\stem(q)^{\frown}\alpha)$ which forces $\dot{g}(\check{l})=\check{\beta}_{\alpha}$ for some $\beta_{\alpha}<\alpha$, since $q\uhr(\stem(q)^{\frown}\alpha)\Vdash\dot{g}(\check{l})<\dot{b}(\check{l})=\check{\alpha}$ and $\alpha<\kappa_l<\kappa_{l+1}$. By the normality of $I_l$, there is $W\subseteq\osucc_{\stem(q)}(q)$ with $W\in I_l^+$ such that $\beta_{\alpha}=\beta$ for some $\beta$ and all $\alpha\in W$. If $q':=\bigcup_{\alpha\in W}q_{\alpha}$, then $q'\leq_0q$ and $q'\Vdash\dot{g}(\check{l})=\check{\beta}$.
		
		We repeat this argument and build a fusion sequence $(p_i)_{i\in\omega}$ such that for all $s\in p_i$ with $n\leq|s|<n+i$ there is an ordinal $\alpha_s$ such that $p_i\uhr s\Vdash\dot{g}(\check{|s|})=\check{\alpha}_s$. Let $p_{\omega}$ be the fusion limit of all $p_n$. Then whenever $s\in p_{\omega}$ with $n\leq|s|<n+i$, there is an ordinal $\alpha_s$ such that $p_i\uhr s\Vdash\dot{g}(\check{|s|})=\check{\alpha}$. Moreover, for each $l$ there are at most $\kappa_{l-1}$ many $s\in p_{\omega}$ with $|s|=l$ and each $\alpha_s$ is in $\kappa_l$, so $\alpha_l:=\sup_{s\in p,\;|s|=l}\alpha_s<\kappa_l$.
		
		Let $h$ be defined by $h(l):=\alpha_l$ for $l\geq n$, $0$ otherwise. Then it follows that $h\in\prod_n\kappa_n$ and $p_{\omega}\Vdash\dot{g}<^*\check{h}$. Moreover, since $(f_{\alpha})_{\alpha<\kappa^*}$ is a scale, $h<^*f_{\alpha}$ for some $\alpha<\kappa^*$. In particular $p_{\omega}\Vdash\dot{g}<^*\check{f}_{\alpha}$, so we are done.
	\end{proof}
	
	\section{The Main Theorem}
	
	In this section we will prove the main theorem:
	
	\begin{mysen}\label{MainThm}
		Assume $\GCH$ holds and $(\kappa_k)_{k\in\omega}$ is an increasing sequence of supercompact cardinals. Let $n\in\omega$, $n\geq 1$. There is a forcing extension in which $\GCH$ holds, $\kappa_0=\aleph_{n+1}$, $(\sup_k\kappa_k)^+=\aleph_{\omega+1}$ and there are stationarily many $\gamma\in\aleph_{\omega+1}\cap\cof(\aleph_n)$ which are good but not approachable.
	\end{mysen}
	
	\begin{mybem}
		Note that in the case $n>1$, the goodness of $\gamma$ is trivial by $\GCH$.
	\end{mybem}
	
	Before we can go through with the proof of Theorem \ref{MainThm}, we will need a result regarding the preservation of non-approachability under sufficiently closed forcing. This result is well-known but we will provide a proof because in our case, only the direct extension ordering is sufficiently closed.
	
	\begin{mylem}\label{ClosurePreserve}
		Assume $\gamma\leq\delta$ are ordinals such that $\cf(\gamma)$ is not a strong limit and $d\colon[\delta]^2\to\omega$ is a subadditive function such that $\gamma$ is not $d$-approachable. Let $(\dP,\leq,\leq_0)$ be a Prikry-type forcing and assume that $(\dP,\leq_0)$ is ${<}\,\cf(\gamma)$-closed. Then $\gamma$ is not $d$-approachable after forcing with $\dP$.
	\end{mylem}
	
	\begin{proof}
		Fix an unbounded and increasing function $F\colon\cf(\gamma)\to\gamma$. By Fact \ref{ApproachRefinement}, after forcing with $\dP$, $\gamma$ is $d$-approachable if and only if there exists an unbounded subset of $\im(F)$ on which $d$ is bounded. By Lemma \ref{PrikryPumpUp} (3) we can decide ordinals below $\cf(\gamma)$ using $\leq_0$. In particular, $\dP$ does not collapse $\cf(\gamma)$.
		
		So assume that $\dot{G}$ is a $\dP$-name for an unbounded and increasing function from $\cf(\check{\gamma})$ into $\im(\check{F})$ such that $\check{d}$ is bounded by $\check{l}$ on the image of $\dot{G}$, forced by some $p\in\dP$. By induction on $\alpha<\cf(\gamma)$, we define a $\leq_0$-decreasing sequence $(p_{\alpha})_{\alpha<\cf(\gamma)}$ with $p_0\leq_0p$ such that for any $\alpha<\cf(\gamma)$ there is some $\chi_{\alpha}$ such that $p_{\alpha}\Vdash\dot{G}(\check{\alpha})=\check{F}(\check{\chi}_{\alpha})$.
		
		Assume $(p_{\beta})_{\beta<\alpha}$ has been defined for $\alpha<\cf(\gamma)$. Let $p_{\alpha}'$ be a $\leq_0$-lower bound of $(p_{\beta})_{\beta<\alpha}$. Then $p_{\alpha}'\Vdash\exists i<\cf(\gamma)(\dot{G}(\check{\alpha})=\dot{F}(i))$. Let $p_{\alpha}\leq_0p_{\alpha}$ force the statement for $i=\check{\chi}_{\alpha}$ for some $\chi_{\alpha}$.
		
		Now let $B:=\{F(\chi_{\alpha})\;|\;\alpha<\cf(\gamma)\}$. We note that $B$ is in the ground model. Since $\dot{G}$ is forced to be increasing and $F$ is actually increasing, $\chi_{\alpha}\geq\alpha$ for every $\alpha<\cf(\gamma)$. In particular, $B$ is unbounded in $\gamma$. It follows that for any $\alpha<\alpha'<\cf(\gamma)$,
		$$p_{\alpha}'\Vdash\check{d}(\check{F}(\check{\chi}_{\alpha}),\check{F}(\check{\chi}_{\alpha'}))=\check{d}(\dot{G}(\check{\alpha}),\dot{G}(\check{\alpha'}))\leq\check{l}$$
		so $d$ is bounded on $[B]^2$ by $l$, a contradiction, as $B\in V$.
	\end{proof}
	
	Now we can prove the main theorem:
	
	\begin{proof}[Proof of Theorem \ref{MainThm}]
		Let $\kappa:=\kappa_0$ and $\kappa^*:=(\sup_k\kappa_k)^+$. Let $l\colon\kappa\to V_{\kappa}$ be a Laver function (see \cite{LaverIndestruct}). We define an Easton-support Magidor iteration $(\dP_{\alpha},\dot{\dQ}_{\alpha})_{\alpha<\kappa}$ with the following iterands:
		
		\begin{enumerate}
			\item Case 1: $\alpha$ is Mahlo, $|\dP_{\beta}|<\alpha$ for every $\beta<\alpha$ and $\dP_{\alpha}$ forces that $l(\alpha)$ is a ${<}\,\alpha$-strategically closed poset which forces that $\LIP(\check{\alpha},\dot{\alpha}^{+k})$ holds for every $k\geq 1$, witnessed by some $\dot{I}_k$. In this case, we let $\dot{\dQ}_{\alpha}$ be a $\dP_{\alpha}$-name for
			$$l(\alpha)*\dot{\dL}((\dot{\alpha}^{+k},\dot{I}_k)_{k\geq 1})*\dot{\Coll}(\check{\aleph}_n,\dot{\alpha}^{+\omega+1})$$
			
			For simplicity, we will replace $\dot{\dL}((\dot{\alpha}^{+k},\dot{I}_k)_{k\geq 1})$ by its subset which is dense in both orderings and ${<}\,\alpha$-closed with regards to $\leq_0$ (see Definition \ref{DefSuborder} and Lemma \ref{SuborderClosed}).
			\item Case 2: Otherwise, we let $\dot{\dQ}_{\alpha}$ be a $\dP_{\alpha}$-name for $\dot{\Coll}(\check{\aleph}_n,\check{\alpha})$.
		\end{enumerate}
		
		Let $\dP_{\kappa}$ be the direct limit of $(\dP_{\alpha},\dot{\dQ}_{\alpha})_{\alpha<\kappa}$ and let $\dP:=\dP_{\kappa}*\prod_{k\in\omega}\dot{\Coll}(\check{\kappa}_k,<\check{\kappa}_{k-1})$.
		
		\setcounter{myclaim}{0}
		
		\begin{myclaim}
			After forcing with $\dP$, all cardinals $\leq\aleph_n$ and $\geq\kappa^*$ are preserved. Every cardinal in the interval $(\aleph_n,\kappa)$ is collapsed to $\aleph_n$ and every $\kappa_k$ is collapsed to $\kappa^{+k}$. Consequently, $\kappa^*$ becomes $\aleph_{\omega+1}$.
		\end{myclaim}
		
		\begin{proof}
			$\dP_{\kappa}$ has the Prikry property by Lemma \ref{PrikryIter}. By induction on $\alpha\leq\kappa$ it follows that $\dP_{\alpha}$ has a ${<}\,\aleph_n$-closed direct ordering (using Lemma \ref{DirectClosureLimit}) and thus preserves all cardinals up to and including $\aleph_n$ (the direct extension ordering on $\dL((\check{\alpha}^{+k},\dot{I}_k)_{k\in\omega})$ is forced to be ${<}\,\check{\aleph}_n$-closed because $\dP_{\alpha}$ forces $\check{\alpha}\geq\check{\aleph}_n$). So after forcing with $\dP_{\kappa}$, all cardinals below and including $\aleph_n$ as well as $\geq\kappa$ are preserved and every regular cardinal in $(\aleph_n,\kappa)$ is collapsed to have size $\aleph_n$.
			
			Moreover, for any $k<n$, $\dP_{\kappa}$ does not add new subsets to $\aleph_k$ by the Prikry property and the closure of the direct extension ordering. Since $\kappa_0$ is collapsed to $\aleph_n$, we obtain that $\dP_{\kappa}$ forces $\GCH$.
			
			Now clearly $\prod_{k\in\omega}\dot{\Coll}(\check{\kappa}_k,<\check{\kappa}_{k-1})$ is forced to preserve cardinals below and including $\check{\kappa}_0$ as well as above and including $\check{\kappa}^*$ (since every ${<}\,\check{\kappa}_k$-sequence is forced to have been added by a $\check{\kappa}_{k+1}$-cc forcing), so the claim follows easily.
		\end{proof}
		
		The more substantial claim is of course showing that $\dP$ forces that there are stationarily many $\gamma\in\aleph_{\omega+1}\cap\cof(\aleph_n)$ which are not approachable. Let $d\colon[\kappa^*]^2\to\omega$ be any subadditive coloring which is $(\kappa_k)_{k\in\omega}$-normal. We will show that there are stationarily many $\gamma\in\aleph_{\omega+1}\cap\cof(\aleph_n)$ which are not $d$-approachable in any extension by $\dP$.
		
		Let $G:=G_{\kappa}*H_{\kappa}$ be $\dP_{\kappa}*\prod_{k\in\omega}\dot{\Coll}(\check{\kappa}_k,<\check{\kappa}_{k-1})$-generic and let $C\in V[G]$ be club in $\kappa^{+\omega+1}$, $C=\dot{C}^{G}$. We will verify that there is a point in $C\cap\cof(\aleph_n)$ which is good but not $d$-approachable. To this end, using that $l$ is a Laver function, let $j\colon V\to M$ be a $\kappa^*$-supercompact embedding with critical point $\kappa$ such that $j(l)(\kappa)$ is a $\dP_{\kappa}$-name for $\prod_{k\in\omega}\dot{\Coll}(\check{\kappa}_k,<\check{\kappa}_{k-1})$.
		
		\begin{myclaim}\label{Claim2}
			There is a filter $G_{\kappa,j(\kappa)}*H_{j(\kappa)}$ such that $G_{\kappa}*H_{\kappa}*G_{\kappa,j(\kappa)}*H_{j(\kappa)}$ is $j(\dP)$-generic over $V$ and in $V[G_{\kappa}*H_{\kappa}*G_{\kappa,j(\kappa)}*H_{j(\kappa)}]$, $j$ lifts to
			$$j\colon V[G_{\kappa}*H_{\kappa}]\to M[G_{\kappa}*H_{\kappa}*G_{\kappa,j(\kappa)}*H_{j(\kappa)}].$$
		\end{myclaim}
		
		\begin{proof}
			Since $j(l)(\kappa)=\prod_{k\in\omega}\dot{\Coll}(\check{\kappa}_k,<\check{\kappa}_{k-1})$ is a $\dP_{\kappa}$-name for a ${<}\,\kappa$-strategically closed (even directed-closed) poset which forces $\LIP(\check{\kappa},\dot{\kappa}^{+k})$ for every $k\geq 1$ (by Lemma \ref{OmegaLIPLemma}), $j(\dP_{\kappa})=\dP*\dot{\dP}_{\kappa,j(\kappa)}$ for some $\dot{\dP}_{\kappa,j(\kappa)}$. So we can let $G_{\kappa,j(\kappa)}$ be any $\dot{\dP}_{\kappa,j(\kappa)}^{G_{\kappa}*H_{\kappa}}$-generic filter  over $V[G_{\kappa}*H_{\kappa}]$. Since $j[G_{\kappa}]=G_{\kappa}\subseteq G_{j(\kappa)}$, the embedding lifts to $j\colon V[G_{\kappa}]\to M[G_{j(\kappa)}]$.
			
			By assumption $H_{\kappa}\in M[G_{j(\kappa)}]$. Since $H_{\kappa}$ is a filter, it is directed. Moreover, $j[H_{\kappa}]\in M[G_{j(\kappa)}]$ by the size of $H_{\kappa}$ and the closure of $M$. Furthermore, $j[H_{\kappa}]$ is a directed subset of $j(\prod_{k\in\omega}\dot{\Coll}(\check{\kappa}_k,<\check{\kappa}_{k-1}))^{G_{\kappa}*H_{\kappa}*G_{\kappa,j(\kappa)}}$. Therefore $q:=\bigcup j[H_{\kappa}]$ is a condition in the aforementioned poset, since that poset is ${<}\,j(\kappa)$-directed closed and $|j[H_{\kappa}]|=\kappa^*<j(\kappa)$. By forcing with $j(\prod_{k\in\omega}\dot{\Coll}(\check{\kappa}_k,<\check{\kappa}_{k-1}))^{G_{\kappa}*H_{\kappa}*G_{\kappa,j(\kappa)}}$ over $V[G_{\kappa}*H_{\kappa}*G_{\kappa,j(\kappa)}]$ below the condition $q$, we obtain $H_{j(\kappa)}$. It follows that $j[G_{\kappa}*H_{\kappa}]\subseteq G_{\kappa}*H_{\kappa}*G_{\kappa,j(\kappa)}*H_{j(\kappa)}$, so $j$ indeed lifts as claimed.
		\end{proof}
		
		Let $G_{j(\kappa)}:=G_{\kappa}*H_{\kappa}*G_{\kappa,j(\kappa)}$. It follows that $\rho:=\sup(j[\kappa^*])\in j(C)$. So all that is left to show is that $\rho$ is not $j(d)$-approachable and that it is good and has cofinality $\aleph_n$ in the model $M[G_{j(\kappa)}*H_{j(\kappa)}]$.
		
		\begin{myclaim}\label{Claim3}
			$\rho$ is not $j(d)$-approachable in $M[G_{j(\kappa)}*H_{j(\kappa)}]$.
		\end{myclaim}
		
		\begin{proof}
			In $M[G_{\kappa}*H_{\kappa}*G_{\kappa,j(\kappa)}]$, $j[\kappa^*]$ is a cofinal subset of $\rho$ and has order-type $\kappa^*$. Moreover, for any $\alpha=j(\alpha')\in j[\kappa^*]$ and $k\in\omega$, we have
			$$\{\beta\in j[\kappa^*]\cap j(\alpha')\;|\;d(\beta,j(\alpha'))\leq k\}=j[\{\beta\in\kappa^*\cap\alpha'\;|\;d(\beta,\alpha')\leq k\}]$$
			and the latter set has size $\leq\kappa_k$ by the $(\kappa_k)_{k\in\omega}$-normality of $d$. It follows that $j(d)$ induces a $(\kappa_k)_{k\in\omega}$-normal coloring $e$ on $\kappa^*$ such that in any extension, $\rho$ is $j(d)$-approachable if and only if $\kappa^*$ is $e$-approachable: Simply let $F\colon\kappa^*\to j[\kappa^*]$ be order-preserving and let $e(\alpha,\beta):=d(F(\alpha),F(\beta))$.
			
			Moreover, $\kappa_k=\kappa^{+k}$ in $M[G_{\kappa}*H_{\kappa}*G_{\kappa,j(\kappa)}]$. Let $H_{\kappa+1}$ be the $\dot{\dL}((\dot{\kappa}^{+k},\dot{I}_k)_{k\geq 1})*\dot{\Coll}(\check{\aleph}_n,\dot{\kappa}^{+\omega+1})^{G_{\kappa}*H_{\kappa}*G_{\kappa,j(\kappa)}}$-generic filter induced by $G_{\kappa,j(\kappa)}$. Then by Theorem \ref{NoNewApproach}, $\kappa^*$ is not $e$-approachable in $M[G_{\kappa}*H_{\kappa}*H_{\kappa+1}]$ and so $\rho$ is not $j(d)$-approachable in the same model. Moreover, in $M[G_{\kappa}*H_{\kappa}*H_{\kappa+1}]$, $\cf(\rho)=\aleph_n$.
			
			We will show that the non-$j(d)$-approachability of $\rho$ is preserved when going to $M[G_{\kappa}*H_{\kappa}*G_{\kappa,j(\kappa)}]$. Let $G_{\kappa+1,j(\kappa)}$ be the $j(\dP)_{\kappa+1,j(\kappa)}$-generic induced by $G_{\kappa,j(\kappa)}$. $j(\dP)_{\kappa+1,j(\kappa)}$ is an Easton-support Magidor iteration of Prikry-type forcings of length $j(\kappa)$ and therefore has the Prikry property. By the same argument as before, the direct extension ordering on $j(\dP)_{\kappa+1,j(\kappa)}$ is ${<}\,\aleph_n$-closed. So by Lemma \ref{ClosurePreserve}, $j(\dP)_{\kappa+1,j(\kappa)}$ preserves the non-$j(d)$-approachability of $\rho$.
			
			Since $(\prod_{k\in\omega}\dot{\Coll}(\check{j(\kappa_k)},<\check{j(\kappa_{k-1})}))^{G_{j(\kappa)}}$ is ${<}\,\aleph_{n+1}$-closed, it also cannot make $\rho$ a $d$-approachable ordinal (since $d$-approachability can of course always be witnessed by sets of minimal order-type which have size $\aleph_n$). So $\rho$ is not $j(d)$-approachable in $M[G_{j(\kappa)}*H_{j(\kappa)}]$.
		\end{proof}
		
		In the proof of the preceding claim we have shown that $\cf(\rho)=\aleph_n$ in $M[G_{j(\kappa)}*H_{j(\kappa)}]$. So in the case $n>1$, the goodness of $\rho$ is clear because the $\GCH$ holds. However, in the other case we have some work to do. Thus, assume $n=1$.
		
		\begin{myclaim}\label{Claim4}
			$\rho$ is good in $M[G_{j(\kappa)}*H_{j(\kappa)}]$.
		\end{myclaim}
		
		\begin{proof}
			Let $H_{\kappa+1}:=H_{\kappa+1}^{\dL}*H_{\kappa+1}^{\dC}$. Let $b$ be the generic branch of $H_{\kappa+1}^{\dL}$. By Lemma \ref{ExactUpperBound}, $b$ is an exact upper bound of $(f_{\alpha})_{\alpha<\kappa^*}$ whenever $(f_{\alpha})_{\alpha<\kappa^*}$ is a $(\kappa^*,(\kappa_n)_{n\in\omega})$-scale.
			
			Let $b^j$ be defined by $b^j(k):=j(b(k))$ and let $j((f_{\alpha})_{\alpha<\kappa^*})=:(f^j_{\alpha})_{\alpha<j(\kappa^*)}$. It is clear that $b^j$ is an exact upper bound of $(j(f_{\alpha}))_{\alpha<\kappa^*}$ and thus of $(f^j_{\alpha})_{\alpha<\rho}$ (since $(j(f_{\alpha}))_{\alpha<\kappa^*}$ is cofinal in $(f^j_{\alpha})_{\alpha<\rho}$). Moreover, since each $I_k$ concentrates on points of cofinality $\kappa_{k-1}$ by Lemma \ref{LIPLemma}, $\cf(b^j(k))=\kappa_{k-1}$ for almost all $k\in\omega$ in $M[G_{\kappa}*H_{\kappa}]$ and therefore $\cf(b^j(k))=\aleph_1$ for almost all $k\in\omega$ in $M[G_{\kappa}*H_{\kappa}*H_{\kappa+1}]$ (since all $\kappa_{k-1}$ are collapsed to $\aleph_1$). Hence (since $\cf(\rho)=\aleph_1$ in that model) $\rho$ is a good point of $(f^j_{\alpha})_{\alpha<\rho}$ in $M[G_{\kappa}*H_{\kappa}*H_{\kappa+1}]$, so there exists an unbounded $A\subseteq\rho$ and $l\in\omega$ such that for every $k\geq l$, $(f_{\alpha}(l))_{\alpha\in A}$ is increasing. This of course remains the case in any further forcing extension. Ergo $\rho$ is a good point in $M[G_{j(\kappa)}*H_{j(\kappa)}]$.
		\end{proof}
		
		In summary, $\rho\in j(C)\cap\cof(\aleph_n)$ and $\rho$ is good but not $j(d)$-approachable. By elementarity, in $V[G]$, there exists $\gamma\in C$ with cofinality $\aleph_n$ such that $\gamma$ is good but not $d$-approachable. This finishes the proof.
		
	\end{proof}
	
	\section{Failure of Approachability without a Bad Scale}
	
	We close by sketching the consistency proof of a model where there is a good scale and the approachability property fails at $\aleph_{\omega+1}$. We use a poset adding a scale with club-many good points. This material can be found in a recent article of the second author and Mildenberger (see \cite[Section 2.3]{LevineMildenbergerGoodScaleNonCompactSquares}).
	
	\begin{mydef}
		Let $\vec{\lambda}=(\lambda_n)_{n\in\omega}$ be an increasing sequence of regular cardinals. Let $\lambda^*:=(\sup_n\lambda_n)^+$. Let $\dG(\vec{\lambda})$ be the poset such that conditions are $(f_{\beta})_{\beta\leq\alpha}$ (for $\alpha<\lambda^*$) such that
		\begin{enumerate}
			\item for every $\beta\leq\alpha$, $f_{\beta}\in\prod_n\lambda_n$,
			\item for every $\beta<\gamma\leq\alpha$, $f_{\beta}<^*f_{\gamma}$,
			\item for every $\beta\leq\alpha$, if $\cf(\beta)>\omega$, then $\beta$ is good for $(f_{\gamma})_{\gamma<\beta}$.
		\end{enumerate}
		ordered by end-extension.
	\end{mydef}
	
	The poset $\dG(\vec{\lambda})$ is ${<}\,\lambda^*$-strategically closed (hence it is ${<}\,\lambda^*$-distributive) and ${<}\,\omega_1$-directed closed. Ergo, if $(\lambda^*)^{<\lambda^*}=\lambda^*$ (which is the case under $\GCH$), it preserves all cofinalities.
	
	We now show that a slight modification of the proof of Theorem \ref{MainThm} yields the following:
	
	\begin{mysen}
		Let $(\kappa_k)_{k\in\omega}$ be an increasing sequence of supercompact cardinals. Let $n\in\omega, n\geq 1$. There is a forcing extension in which $\GCH$ holds, $\kappa_0=\aleph_{n+1}$, $(\sup_k\kappa_k)^+=\aleph_{\omega+1}$, there is a good scale on $(\aleph_k)_{k\geq n+1}$ and $\aleph_{\omega+1}\cap\cof(\aleph_n)\notin I[\aleph_{\omega+1}]$.
	\end{mysen}
	
	\begin{proof}
		We let $\dP_{\kappa}$ be as in the proof of Theorem \ref{MainThm} and we define
		$$\dP:=\dP_{\kappa}*\prod_{k\in\omega}\dot{\Coll}(\check{\kappa}_k,<\check{\kappa}_{k+1})*\dot{\dG}((\check{\kappa}_k)_{k\in\omega})$$
		
		Let $G$ be $\dP$-generic, write $G=G_{\kappa}*H_{\kappa}^{\dC}*H_{\kappa}^{\dG}$. We will verify that $V[G]$ works.
		
		Since $\GCH$ holds in $V[G_{\kappa}*H_{\kappa}]$, it continues to hold in $V[G]$ by the strategic closure and size of $\dot{\dG}((\kappa_k)_{k\in\omega})^{G_{\kappa}*H_{\kappa}}$. Moreover, no cardinals are collapsed by the latter forcing and there clearly is a good scale on $(\kappa_k)_{k\in\omega}=(\aleph_{n+k+1})_{k\in\omega}$.
		
		So the only thing that remains to be shown is that there are stationarily many non-approachable points of cofinality $\aleph_n$. To this end, fix a $(\kappa_k)_{k\in\omega}$-normal coloring on $\kappa^*$ in $V$ and let $C\in V[G]$ be club in $\kappa^*$. Let $j\colon V\to M$ be a $\kappa^*$-supercompact embedding such that $j(l)(\kappa)$ is a $\dP_{\kappa}$-name for $\prod_{k\in\omega}\dot{\Coll}(\check{\kappa}_k,<\check{\kappa}_{k+1})*\dot{\dG}((\check{\kappa}_k)_{k\in\omega})$ and define $\rho:=\sup(j[\kappa^*])$.
		
		\setcounter{myclaim}{0}
		
		\begin{myclaim}
			There is a filter $G_{\kappa,j(\kappa)}*H_{j(\kappa)}^{\dC}*H_{j(\kappa)}^{\dG}$ such that $G^j:=G*G_{\kappa,j(\kappa)}*H_{j(\kappa)}^{\dC}*H_{j(\kappa)}^{\dG}$ is $j(\dP)$-generic over $V$ and the embedding $j$ lifts in $V[G^j]$ to $j\colon V[G]\to M[G^j]$.
		\end{myclaim}
		
		\begin{proof}
			By Claim \ref{Claim2} in the proof of Theorem \ref{MainThm} we know that we can find $G_{\kappa,j(\kappa)}*H_{j(\kappa)}^{\dC}$ such that $j$ lifts to
			$$j\colon V[G_{\kappa}*H_{\kappa}^{\dC}]\to M[G*G_{\kappa,j(\kappa)}*H_{j(\kappa)}^{\dC}]$$
			since in $M[G]$, $j(l)(\kappa)$ is a ${<}\,\kappa$-strategically closed poset which forces $\LIP(\check{\kappa},\check{\kappa}^{+k})$ for every $k\geq 1$ (the product of the collapses forces $\LIP$ and $\dot{\dG}$ is forced to preserve it by its distributivity). So we only need to construct $H_{j(\kappa)}^{\dG}$. Define $q':=\bigcup j[H_{\kappa}^{\dG}]$. Then $q'=(f_{\alpha})_{\alpha<\rho}$ is almost a condition in $\dot{\dG}((j(\kappa_k))_{k\in\omega})^{G*G_{\kappa,j(\kappa)}*H_{j(\kappa)}^{\dC}}$ save for the fact that it is not of successor length. So we have to find a suitable $f_{\rho}$.
			
			By Claim \ref{Claim4} in the proof of Theorem \ref{MainThm}, $\rho$ is a good point of $(f_{\alpha})_{\alpha<\rho}$. Ergo there exists an exact upper bound $f_{\rho}$ of $(f_{\alpha})_{\alpha<\rho}$ such that $\cf(f_{\rho}(k))=\cf(\rho)$ for almost all $k\in\omega$. Then we can just let $q:=(q')^{\frown}f_{\rho}$ and see that indeed $q\in\dot{\dG}((j(\kappa_k))_{k\in\omega})^{G*G_{\kappa,j(\kappa)}*H_{j(\kappa)}^{\dC}}$. Now we can just let $H_{j(\kappa)}^{\dG}$ be any generic filter containing $q$. Since $j[G]\subseteq G*G_{\kappa,j(\kappa)}*H_{j(\kappa)}^{\dC}*H_{j(\kappa)}^{\dG}$ by construction, $j$ lifts as required.
		\end{proof}
		
		It follows that $\rho\in j(C)$. So we are done after showing:
		
		\begin{myclaim}
			$\rho$ is not $j(d)$-approachable in $M[G^j]$.
		\end{myclaim}
		
		\begin{proof}
			By Claim \ref{Claim3} in the proof of Theorem \ref{MainThm}, $\rho$ is not $j(d)$-approachable in $M[G*G_{\kappa,j(\kappa)}*H_{j(\kappa)}^{\dC}]$ and has cofinality $\aleph_n$ there. Since $H_{j(\kappa)}^{\dG}$ is generic for a ${<}\,j(\kappa^*)$-strategically closed forcing, $j(\kappa^*)=\aleph_{\omega+1}$ in $M[G*G_{\kappa,j(\kappa)}*H_{j(\kappa)}^{\dC}]$ and the approachability of $\rho$ would always be witnessed by a set of order-type $\aleph_n$, $\rho$ remains non-$j(d)$-approachable in $M[G^j]$.
		\end{proof}
		
		Moreover, $\cf(\rho)=\aleph_n$ in $M[G^j]$ as well by the strategic closure. By elementarity, in $V[G]$, there is $\gamma\in C$ with cofinality $\aleph_n$ such that $\gamma$ is not $d$-approachable. This finishes the proof.
		
	\end{proof}
	
	\printbibliography
\end{document}